\documentclass[10pt,twoside,reqno]{amsart}
\usepackage{amsmath,amsfonts,amssymb,amsthm}
\usepackage{enumerate}
\usepackage[utf8]{inputenc}
\usepackage{hyperref}
\usepackage[usenames]{color}
\usepackage[english]{babel}
\usepackage{graphicx}
\usepackage{tikz-cd}
\usepackage{epstopdf}
\usepackage{subcaption}
\usepackage[margin=2.5cm]{geometry}

\newtheorem{thm}{Theorem}[section]

\newtheorem{definition}[thm]{Definition}
\newtheorem{prop}[thm]{Proposition}
\newtheorem{proposition}[thm]{Proposition}

\newtheorem{corollary}[thm]{Corollary}

\newtheorem{lemma}[thm]{Lemma}

\newtheorem*{thmmain}{Main Theorem}

\theoremstyle{remark}

\theoremstyle{remark}
\newtheorem{remark}[thm]{Remark}

\newtheorem{example}[thm]{Example}

\newcommand{\R}{\mathbb{R}}
\newcommand{\Q}{\mathbb{Q}}
\newcommand{\N}{\mathbb{N}}

\newcommand{\C}{\mathbb{C}}

\newcommand{\OO}{\mathcal{O}}

\newcommand{\DD}{\mathcal{D}}

\newcommand{\um}{{\underline m}}
\newcommand{\un}{{\underline n}}
\newcommand{\supp}{\operatorname{supp}}

\newcommand{\x}{\mathbf{x}}
\newcommand{\xx}{\frac{\partial}{\partial x}}
\newcommand{\yy}{\frac{\partial}{\partial y}}
\newcommand{\zz}{\frac{\partial}{\partial z}}

\newcommand{\ad}{\operatorname{ad}}

\newcommand{\ku}{^{(k)}}

\newcommand{\nobracket}{}

\newcommand{\tmem}[1]{{\em #1\/}}
\newcommand{\tmmathbf}[1]{{\boldsymbol{#1}}}
\newcommand{\tmop}[1]{{\operatorname{#1}}}
\newcommand{\tmtextbf}[1]{\text{{\bfseries{#1}}}}

\newenvironment{enumeratealpha}{\begin{enumerate}[a{\textup{)}}] }{\end{enumerate}}
\newenvironment{enumeratenumeric}{\begin{enumerate}[1.] }{\end{enumerate}}
\newenvironment{itemizedot}{\begin{itemize} }{\end{itemize}}

\newcommand{\Gr}{\operatorname{Gr}}
\newcommand{\ord}{\mathrm{ord}}

\newcommand{\tO}{\widetilde{\mathcal{O}}}
\newcommand{\oO}{\mathcal{O}}
\newcommand{\cD}{\mathcal{D}}
\newcommand{\tcD}{\widetilde{\cD}}
\newcommand{\vO}{\mathcal{O}}

\usepackage[foot]{amsaddr}

\title{Bruno ideal and the variety of centers for singular germs of vector fields}
\author{María Martín-Vega}
\address[1]{Université Paris Cité, Sorbonne Université, CNRS, IMJ-PRG, F-75013 Paris, France.}
\email[1]{martinvega@imj-prg.fr}
\author{Daniel Panazzolo}
\address[2]{ D\' epartement de Math\' ematiques -- IRIMAS--UHA,
 18 Rue des Fr\`eres Lumi\`ere, 68093 Mulhouse, France. }
\email[2]{daniel.panazzolo@uha.fr}
\date{}

\begin{document}

\maketitle

\begin{abstract}
    Given a germ of singular analytic vector field $\partial$, we consider the locus where $\partial$ is formally equivalent to its linear part. In the logarithmic case, this locus is defined by the vanishing of a formal ideal $B(\partial)$.  We show that, under convenient arithmetic conditions on the linear part of $\partial$, such ideal $B(\partial)$ is in fact analytic. As a consequence, its zero set $V$ is an analytic invariant variety and the foliation defined by the restriction $\partial|_V$ is analytically linearizable. 
\end{abstract}

\tableofcontents

\section{Introduction}

Based on ideas of A.D. Bruno, the aim of this paper is to study the existence of analytic invariant sets related to the convergence of the Poincaré-Dulac normalization for germs of singular analytic vector fields. 

Let us start by briefly recalling some basic concepts and results of normal forms. The detailed definitions will be given in the next subsection.
A germ of singular analytic vector field $\partial$ is called a {\em Poincaré-Dulac normal form} (or simply {\em normal form}) if we can decompose it in the form
$$
\partial = S + R
$$ 
where:
\begin{itemize}
\setlength{\itemsep}{0.4em}
\item[(1)] $S= \sum_{i=1}^n \lambda_ix_i \frac{\partial}{\partial x_i}$ is linear diagonal vector field
\item[(2)] $R$ is a non-linear vector field with a nilpotent linear part.
\item[(3)] $S$ and $R$ commute, i.e.~the Lie bracket $[S,R]$ vanishes. 
\end{itemize}
As remarked by Martinet in \cite{Martinet1981Norm}, these properties are equivalent to say that $S$ and $R$ are respectively the semi-simple and nilpotent components in the {\em Jordan decomposition} of $\partial$ (see subsection \ref{subsect-jordan-decomp}).  This definition extends to the case where $\partial$ is a singular {\em formal} vector field, i.e.~to a vector field of the form $\partial = \sum f_i \frac{\partial}{\partial x_i}$ where each component $f_i$ is a formal series without constant term.

Let us express the condition $[S,R]=0$ more explicitly. Using the logarithmic basis $\{  x_i \frac{\partial}{\partial x_i}  \}_{i=1}^n$ for vector fields, we can expand 
$R$ in the form 
$$\sum_{i=1}^n \sum_{\um\in M_i } a_{\um}\x ^{\um}  x_i \frac{\partial}{\partial x_i} \, ,$$ where $M_i$ is a subset of $\mathbb{Z}_{\ge 0}^{i-1}\times \mathbb{Z}_{\ge -1} \times \mathbb{Z}_{\ge 0}^{n-i-1}$.
The condition $[S,R]=0$ requires that the non-zero terms in the above expansion of $R$ are such that
$S(\x^{\um})=0$, where $S$ acts on $\mathbb{C}[[x]]$ as a derivation. In other words, that $\langle \lambda, \um\rangle=\sum \lambda_i m_i$ vanishes. These terms are called \emph{resonant} and the relations $\langle \lambda, \um\rangle=0$ are called \emph{resonances}. We say a resonance is \emph{negative} if one of the components $m_{i}$ of $\um$ is equal to $-1$, and \emph{positive} otherwise.  

If there is no monomial $\um\in M_1\cup\cdots\cup M_n$ satisfying
\[
\langle\lambda,\um\rangle=0,
\]
we say that $S$ is \emph{non-resonant}. Since $\lambda_1,\ldots,\lambda_n$ are intrinsically determined as the eigenvalues of the linear part of $\partial=S+R$ at the origin, we also say that $\partial$ is \emph{non-resonant}, or that $\partial$ has no resonances.

From the classical results by H.~Poincaré~\cite{Poincare1881} and H.~Dulac~\cite{Dulac1903}, it is known that every singular analytic vector field $\partial$ is formally conjugate to a normal form. In other words, there exists a formal change of coordinates, which we identify to an automorphism $\Phi$ of the ring of formal series, such that 
$$
\Phi \partial \Phi^{-1} = S + R
$$
where $S,R$ are {\em formal} vector fields satisfying the above conditions (1),(2) and (3). In this case, we say that $\Phi \partial \Phi^{-1}$ is a {\em formal normal form} of $\partial$ and that the conjugating automorphism defines a formal {\em normalization}.  

We emphasize that a normal form is not uniquely determined by $\partial$. For this reason, some authors prefer to refer to it as a {\em pre-normal form}. More precisely, although the diagonal vector field $S$ is intrinsically determined by the eigenvalues $\partial$ at the origin, the nilpotent part $R$ is highly non-unique. Indeed, if $\Psi$ is any formal automorphism preserving the diagonal part, that is, satisfying
$$\Psi S \Psi^{-1}=S,$$
then both $S+R$ and $S+\Psi R\Psi^{-1}$ are formal normal forms of the same vector field. The problem of determining {\em unique} normal forms, namely, of explicitly selecting a canonical representative in each formal conjugacy class, remains open.

The analogous problem of obtaining an {\em analytic} normalization presents two main difficulties. Either the normal form itself is not analytic—reflecting the fact that the formal series expansion of $R$ given above has zero radius of convergence—or the conjugating automorphism $\Phi$ fails to be analytic.

Let us now discuss a remarkable result on analytic normalisation proved by Bruno in \cite{Bruno1971_1972}.   
We start by introducing the following notation. To a vector of complex numbers $\lambda = (\lambda_1,\ldots,\lambda_n)\in \C^n$, 
we associate the sequence $\omega_k = \omega_k(\lambda)$, indexed by $k\in \N$, and defined by
\[
\omega_k \;=\; \min \bigl\{\, |\langle \lambda, \underline{m} \rangle| \;:\;
\underline{m} \in M_1 \cup \cdots \cup M_n,\ 
\underline{m} \text{ non-resonant},\ 
\|\underline{m}\| \le 2^k \bigr\},
\]
where we note $\| \um \| = \sum_{i=1}^{n} |m_i|$ and the sets $M_i$ are defined as above. 

We now state the following conditions on a germ of singular vector field $\partial$:
\begin{itemize}
\item[] \textbf{$\omega$-condition}:
We say that  $\partial$ satisfies the  {\em Bruno's arithmetic $\omega$-condition} (or simply the $\omega$-condition) if it has a formal normal form $S+R$ with semi-simple part
$S = \sum_{i=1}^n \lambda_ix_i \frac{\partial}{\partial x_i}$ such that the sequence $\omega_k=\omega_k(\lambda)$ satisfies
\[ 
\sum_{k=0}^{\infty} -\frac{\log \bigl( \omega_k\bigr)}{2^k} < \infty.
\]
\item[] \textbf{Geometric $A$-condition}: We say that  $\partial$ satisfies the {\em geometric A-condition} if it has a formal normal form $S+R$ such that 
$$R=f S$$
for some formal series $f \in \mathbb{C} [[\x]]$. In other words, the nilpotent component is {\em collinear} to the semi-simple component. 
\end{itemize}
Observe that both conditions are formulated in terms of a normal form of $\partial$, and it is therefore natural to ask whether they depend on the particular choice of such a normal form.

The $\omega$-condition is clearly independent of this choice, since it depends only on the spectrum $\lambda=(\lambda_1,\ldots,\lambda_n)$, which is invariant under formal coordinate changes. As we shall prove later, the geometric $A$-condition is likewise independent of the choice of normal form, as a straightforward consequence of the uniqueness of the Jordan decomposition. Hence, both the $\omega$-condition and the geometric $A$-condition are intrinsic properties of $\partial$.
\begin{thm}[Bruno, 1971] \label{thm:Bruno-normalization}
    Let $\partial$ be a germ of singular analytic vector field fulfilling the $\omega$-condition and the geometric $A$-condition. Then, $\partial$ is analytically conjugated to a normal form. Explicitly, up to an analytic coordinate change, $\partial$ admits the decomposition $\partial = S + R$, where $S$ and $R$ satisfy conditions (1)–(3) above.
\end{thm}
Note that this result can be rephrased as follows. Let $\mathcal{A}$ denote the group of germs of analytic automorphisms (i.e., local analytic coordinate changes). Then the orbit of $\partial$ under conjugation,
\[
\mathcal{A}\cdot\partial
:=
\left\{
\Phi\partial\Phi^{-1}:\Phi\in\mathcal{A}
\right\},
\]
contains at least one vector field in normal form.
 \begin{remark}
We remark that the geometric $A$-condition formulated here is strictly stronger than the original $A$-condition formulated by Bruno, but it is better adapted to the formulation of our main result. We refer the reader to \cite{Bruno1971_1972} or to \cite[sec. 10.3]{Stolovitch2000} for the precise statement of the $A$-condition.
\end{remark}
The result of Bruno has been generalized in several different directions. For instance, L. Stolovitch proved a normalization result for foliations generated by perturbations of commuting semi-simple vector fields under conditions of complete integrability. C. Chavaudret \cite{Chavaudret2016} proved an analog of Bruno's result for the normalization of a vector field in the vicinity of a $n$-torus. 

We observe that, if $\partial$ has no-resonances, the geometric $A$-condition is automatically satisfied and the $\omega$-condition implies that $\partial$ is analytically conjugated to a linear diagonal vector field. In this case, we simply say that $\partial$ is {\em analytically linearizable}. 

In this context, the Theorem \ref{thm:Bruno-normalization}
represents a significant weakening of the arithmetic conditions for linearizability previously imposed by other authors, such as C.-L. Siegel~\cite{Siegel1954} and V.-A. Pliss~\cite{Pliss1965}.  We also note that such condition is satisfied for all couples $\lambda \in \mathbb{C}^n$ outside a set of zero Lebesgue measure (see e.g. \cite{AkramovRakhimov2025}),  and that Perez-Marco and Yoccoz~\cite{Perez-Marco-Yoccoz94}  proved it to be a necessary and sufficient condition for analytic linearizability in dimension $n\le 2$. 

On the other hand, if $\partial$ has resonances, the geometric $A$-condition is rarely satisfied and 
imposes a strong restriction in order to guarantee analytic normalization.

In this setting, Bruno's insight, which we want to investigate in this paper, can be summarized as follows: rather than imposing the geometric $A$-condition from the outset, one should consider the problem of analytic linearizability restricted to the {\em formal locus} where this condition holds.

To enunciate our main result, we need introduce the notion of {\em Bruno ideal}. In what follows, we will always assume that $\partial$ is a {\em logarithmic vector field}, i.e.~a vector field preserving the coordinate hyperplanes $x_i = 0$, for $i=1,\ldots,n$. 

Suppose initially that \(\partial\) is a (possibly formal) logarithmic vector field in
normal form with non-vanishing semi-simple part \(S\). Then \(\partial\) can be
written as
\[
\partial = (1+f)\,S + \sum_{j=1}^{n-1} g_j\, T_j,
\]
where
\[
T_j = \sum_{i=1}^n \mu_{ji}\, x_i \frac{\partial}{\partial x_i}
\]
is an arbitrarily chosen family of diagonal vector fields such that the set
\(\{S, T_1, \ldots, T_{n-1}\}\) is \(\mathbb{C}\)-linearly independent, and where
\(f, g_j \in \mathbb{C}[[x]]\) are formal power series without constant term. In more details, let $A=(a_{ij})$ be the $n\times n$ constant matrix defining the change of basis from the standard logarithmic basis $\{x_i \frac{\partial}{\partial x_i}\}_{i=1}^n$ to $\{T_0, T_1, \ldots, T_{n-1}\}$, with $T_0=S$. We can then express $\partial = S+R$ as
\begin{align*}
\partial &= T_0 + \sum_{i=1}^{n} f_i x_i \frac{\partial}{\partial x_i} \\
         &= T_0 + \sum_{i=1}^n f_i \sum_{j=0}^{n-1} a_{i j} T_j \\
         &= \left( 1+ \sum_{i=1}^n a_{i0} f_i \right) T_0 + \sum_{j=1}^{n-1} \left( \sum_{i=1}^n a_{ij} f_i \right) T_j\\
         &= (1+f)\,S + \sum_{j=1}^{n-1} g_j\, T_j.
\end{align*}
Using this expression, the normal form condition~(3) stated above is equivalent to requiring
\[
S(f) = S(g_j) = 0, \qquad j = 1, \ldots, n-1.
\]
In this setting, we define the \emph{Bruno ideal} of \(\partial\) by
\[
B(\partial) = \langle g_1, \ldots, g_{n-1} \rangle \subset \mathbb{C}[[x]],
\]
i.e. the ideal of formal series generated by $g_1,...,g_{n-1}$. 
\begin{remark}
Note that the ideal $B(\partial)$ is independent of the choice of elements $\{T_1, \ldots, T_{n-1}\}$ completing $S$ to a logarithmic basis. Indeed, if $\{T_1', \ldots, T_{n-1}'\}$ is another choice, yielding the expansion $\partial = (1+f)S + \sum_{j=1}^{n-1} g_j' T_j'$, a direct computation using the change-of-basis matrix shows that $g_1', \ldots, g_{n-1}'$ generate the same ideal as $g_1, \ldots, g_{n-1}$.
\end{remark}

More generally, for an arbitrary singular vector field \(\partial\), we define
\(B(\partial)\) as follows: Choose a (formal) conjugated normal form
\(\delta = \Phi \partial \Phi^{-1}\), and define \(B(\partial)\) as the pullback
of \(B(\delta)\) under the coordinate change induced by \(\Phi\). As we shall prove, it turns out
that the resulting formal ideal \(B(\partial)\) is independent of the choice of
the conjugated normal form and is therefore intrinsically attached to
\(\partial\). Additionally, such formal ideal is invariant by $\partial$, i.e. 
$$ \partial \big( B(\partial) \big ) \subset B(\partial)$$
and we can consider the problem of reduction to normal form {\em modulo $B(\partial)$}.  Here, we refer the reader to subsection~\ref{subsect-normalformsdef} for the precise definition of a normal form modulo an invariant ideal. 

Our main result can now be stated as follows:
\begin{thmmain}\label{thm:main}
		Let $\partial$ be an analytic logarithmic vector field satisfying the $\omega$-condition. Then, the ideal $B(\partial)$ is analytic and there exists an analytic automorphism $\Phi$ such that the conjugated derivation $\delta=\Phi\partial\Phi^{-1}$ is in normal form {\em modulo $B(\delta)$}. 
	\end{thmmain}
\noindent We recall that a formal ideal $I \subset \C[[x]]$ is called {\em analytic} if it is generated by analytic germs, i.e. if there exists {\em convergent} series $f_1,...,f_m \in \C\{x\}$ such that 
$$
I = \langle f_1, \ldots, f_{m} \rangle \, \mathbb{C}[[x]].
$$
where the notation in the right-hand-side means the ideal generated by $f_1, \ldots, f_{m}$ in the ring $\mathbb{C}[[x]]$. Note that this property is equivalent to require that $I = \big( I \cap \C\{x\} \big)\, \C[[x]]$. 

It follows from the Theorem that the vanishing locus of $B(\partial)$ defines a germ of analytic subvariety $V(B(\partial)) \subset (\C^n,0)$, which is moreover invariant by $\partial$. We say such variety is the \emph{Bruno variety} of $\partial$.

The following result is an important {\em dynamical} consequence of the theorem.
    \begin{corollary}\label{cor:lienarfoliation}
        Let $\partial$ be an analytic logarithmic vector field satisfying the $\omega$-condition. Then, $\partial$ defines a linear foliation in restriction to its Bruno variety.
    \end{corollary}

\subsection{Previous works}
In~\cite{Bruno1975} and, more recently, in~\cite{Bruno1989}, Bruno introduced two
formal ideals \(\mathcal B(\partial) \supseteq \mathcal A(\partial)\) that are
closely related to the ideal \(B(\partial)\) defined above. 

In these references, he asserts that, if the eigenvalues of $S$ (the semi-simple part of $\partial$) are pairwise commensurable then $\mathcal A(\partial)$ is analytic. More generally, if these eigenvalues are purely imaginary and  satisfy the \(\omega\)-condition, then ideal \(\mathcal B(\partial)\) is analytic.
Unfortunately, these assertions are not proved in the cited works, and we have
not been able to locate any complete proofs in the references quoted by Bruno. 

Note that, under the additional assumption that \(\partial\) is
logarithmic, the ideal \(\mathcal B(\partial)\) originally defined by Bruno coincides with the ideal
\(B(\partial)\) as defined above.

The Main Theorem is also closely related to a result of Stolovitch~\cite{Stolovitch1994}
on the existence of analytic invariant varieties, which itself generalizes to
higher dimensions a classical result of Dulac~\cite{Dulac1923}.  

More precisely, let us state Stolovitch’s result in the present logarithmic
setting. Suppose that all resonant monomials associated with the diagonal vector
field \(S\) belong to the semigroup generated by the monomials
\(x^{\um_1},\ldots,x^{\um_s}\), for some vectors
\(\um_1,\ldots,\um_s \in \mathbb{N}^n \setminus \{0\}\).
Assume moreover that \(S\) satisfies the \(\omega\)-condition. Then, Stolovitch proves that, up to
analytic conjugacy,  any analytic logarithmic vector field with semi-simple part
\(S\) can be expanded in the form
$$
\partial 
= S + \sum_{i=1}^n \left(  \sum_{j=1}^s g_{ij}(x)\, x^{\um_j} \right)
x_i \frac{\partial}{\partial x_i},
$$
where the  \(g_{ij}\) are analytic germs.

Note that Stolovitch’s result remains valid without the assumption that
\(\partial\) is logarithmic, provided instead that all resonances are positive. We refer the reader to the original paper for the detailed definitions.
We also emphasize that the above expansion is particularly
useful for controlling the growth of solution curves of the vector field, since
it allows one to view them as \emph{perturbations} of the resonant monomials.
Indeed, for each \(j\), one has \(\partial(x^{\um_j})\) lying in the
ideal generated by
\(\langle x^{\um_1},\ldots,x^{\um_s} \rangle\). We observe however that the above expansion is not a {\em normal form} in the sense (1),(2),(3) since the non-linear part does not necessarily commute with $S$. 

In the logarithmic setting, our main result implies Theorem~\ref{thm:Bruno-normalization}.
In fact, the proof will follow closely the inductive scheme originally introduced by Bruno in \cite{Bruno1971_1972}, 
but conveniently adapted and reformulated using the elegant approach developed by Martinet in his Bourbaki 
survey~\cite{Martinet1981Norm}. In that survey, Martinet provides a detailed sketch 
of the proof of Bruno’s theorem in the non-resonant case. His method relies on the 
Jordan decomposition and on the natural \(S\)-graduation of the space of vector 
fields, which significantly simplifies the analytic estimates required to 
establish the analyticity of the conjugating map.

During the preparation of this work, we became aware of a recent preprint by 
Romanovski and Walcher~\cite{RomanovskiWalcher2025}, in which the authors propose 
a simplified proof of Theorem~\ref{thm:Bruno-normalization}. Their approach makes 
systematic use of the natural homogeneous \(S\)-graduation, in the same spirit 
as Martinet’s method, but without using the formalism of Jordan decompositions.

In \cite{Ecalle1992}, Ecalle uses the powerful formalism of mould-comould expansions and arborifications to give an alternative proof of Theorem~\ref{thm:Bruno-normalization} in the non-resonant case (see also \cite{FauvetMenousSauzin2025}). It would be interesting to study if Ecalle's formalism can be adapted to prove our main theorem. 
\subsection{Plan of the paper} In Section~\ref{sec:prelim}, we introduce the basic concepts and definitions, as
well as several classical results that will be used in the proof of
the Main Theorem. We also prove that the first part of the Main Theorem (analyticity of $B(\partial)$) is indeed a consequence of the second part (existence of an analytic normalization modulo $B(\partial)$).

The proof of the theorem is divided into two clearly distinguished parts,
presented in Sections~\ref{sec:formalNormalizationModBruno} and 
\ref{sec:analyticreduction}. In
Section~\ref{sec:formalNormalizationModBruno}, we establish a purely formal
normalization result by developing the inductive procedure required to reduce
a formal vector field to a normal form modulo its Bruno ideal.

In Section~\ref{sec:analyticreduction}, we show that, under the
\(\omega\)-condition, it is possible to control the domains of analyticity of
the successive coordinate changes constructed in the previous section when
they are applied to an analytic vector field.

Finally, in Section~\ref{sec:examplesapplications}, we discuss several examples
of Bruno ideals and present some applications of the main result.

\section{Derivations and normal forms modulo ideals}\label{sec:prelim}
In this section, we present some basic definitions and results that will be used in the paper. Since the normal form results will be explained in the formalism of a Jordan decomposition, we will hereby consider vector fields as {\em derivations}.
\subsection{Some facts about free modules and derivations}\label{subsec:derivfreemodule}
We recall some useful facts from commutative algebra and refer the reader to \cite{Eisenbud1995} for a comprehensive treatment. 

Let $R$ be a ring such that $\Q\subset R$ and let $M$ be a free $R$-module. 
\begin{enumerate}
    \item[(M.1)] For each ideal \ $I \subset R$, the quotient module $M / I M$ is a free module over the ring $R / I$.
    \item[(M.2)] The second exterior power $M \wedge M$ is a free module over $R$. 
\end{enumerate}
The {\tmem{coefficient ideal}} of an element $m \in M$ is the ideal
$\Gamma (m) \subset R$ generated by the coefficients of the expansion of $m$
with respect to any arbitrary free basis $\{ e_{\alpha} \}_{\alpha \in A}$ of
$M$. This ideal is independent of the choice of basis.

The {\tmem{ideal of collinearity}} of two elements $m, n \in M$ is the ideal
$\Gamma (m \wedge n)$, i.e. the coefficient ideal of $m \wedge n$, seen as an
element of the second exterior power $M \wedge M$.

Consider now the  $R$-module $\tmop{Der} (R)$ of derivations on $R$. Given an ideal $I\subset R$, we
denote by $\tmop{Der} (- \log I)$ the submodule of derivations $\partial \in
\tmop{Der} (R)$ such that $\partial (I) \subset I$. In this case, we will also
say that $I$ is {\tmem{{$\partial$}-invariant}} or that $I$ is {\tmem{preserved
by}} $\partial$. Note that $\tmop{Der}(-\log I)$ contains the submodule
$I\,\tmop{Der}(R)$ generated by all derivations of the form
$g\partial$, where $g\in I$ and $\partial\in\tmop{Der}(R)$.

A derivation $\partial \in \tmop{Der} (- \log I)$ induces a derivation in the
quotient ring $A = R / I$, which we will denote by $\partial_A$. Given an
element $f \in R$, we can write $(f \partial)_A = f_A \partial_A$, where $f_A$
denotes the class of $f$ in $A$.  
The following two facts will also be useful (see for instance~\cite{Seidenberg67} for a proof):
\begin{enumerate}
    \item[(D.1)] Suppose that $\partial_1, \partial_2 \in \tmop{Der} (- \log I)$ are such
that $(\partial_1)_A = (\partial_2)_A$ in $A = R / I$. Then $\partial_1 -
\partial_2$ belongs to the submodule $I \tmop{Der} (R)$.
\item[(D.2)] If $\partial \in \tmop{Der} (- \log I)$ then $\partial (\tmop{rad} (I))
\subset \tmop{rad} (I)$, where $\tmop{rad}(I)$ denotes the radical of $I$. Moreover, if $I$ is radical and $I = P_1 \cap \cdots
\cap P_n$ is its minimal prime decomposition, then $\partial (P_i) \subset P_i$
for each $i$. 
\end{enumerate}

\subsection{Formal and analytic logarithmic derivations. Graduations and Collinearity}\label{subsect-ideals-expansions}

We denote by $\mathcal{O} =\mathbb{R} \{ x \}$ or $\mathbb{C} \{ x \}$ the
ring of germs of real (resp. complex) analytic functions in $n$ variables
$x_1, \ldots, x_n$ with maximal ideal $\mathfrak{m} = (x_1, \ldots, x_n)$.

Let $\widetilde{\mathcal{O}} =\mathbb{R} [[\nobracket x]] \nobracket$ or
$\mathbb{C} [[x]]$ be the completion of $\mathcal{O}$ with respect to the
Krull topology defined by $\{ \mathfrak{m}^k \}_{k \in \mathbb{N}}$.
Given an ideal $I \subset \mathcal{O}$, we denote by $I
\widetilde{\mathcal{O}}$ the ideal generated by $I$ in $\widetilde{\mathcal{}
\mathcal{O}}$. Since $\mathcal{O} \subset \widetilde{\mathcal{O}}$ is a
faithfully flat extension, we have
\[ I \widetilde{\mathcal{O}} \cap \mathcal{O} = I . \]
We say that an ideal $\tilde{J} \subset \widetilde{\mathcal{O}}$ is
{\tmem{analytic}} if it has the form $\tilde{J} = I \widetilde{\mathcal{O}}$
for some $I \subset \mathcal{O}$.
Each element $f \in \tO$ is defined by a formal series
$$f = \sum_{\underline{m}\in \N^n} a_{\um} x^{\underline{m}}\, ,$$
where $a_\um \in \mathbb{C}$ and $x^\um=x_1^{m_1} \cdots x_n^{m_n}$. We define the support of $f$ by 
$$\supp(f) = \{ {\um}:a_{\um} \ne 0\}\, ,$$ 
and define the order and degree of $f$ by
\begin{equation}\label{def-degree-order}
\ord(f) = \inf\{\|\um\|: \um \in \supp(\partial)\}\quad \text{ and }\quad \deg(f) = \sup\{\|\um\|: \um \in \supp(\partial)\},
\end{equation}
where $\|\um\| = \sum_i |m_i|$. As usual, we adopt the convention that $\deg(0)=-\infty$ and $\ord(0)=\infty$. Note that $f$ is a polynomial iff $\deg(f) < \infty$, or equivalently, if $\supp(f)$ is finite. 

We denote by $\tmop{Der} \left( \oO \right)$ the module of derivations in $\oO$. An element of 
$\tmop{Der} \left( \oO \right)$ is fully determined by its action on the variables $x_1,\ldots,x_n$. Hence, we can identify $\tmop{Der} \left( \oO \right)$ with 
the free module of vector fields
\[ \sum_{i = 1}^n a_i \frac{\partial}{\partial x_i} \]
with coefficients $a_i \in \oO$. 
A derivation $\partial \in \tmop{Der} \left( \oO \right)$ is called {\tmem{local}} if $\partial
(\mathfrak{m}) \subset \mathfrak{m}$, which is equivalent to the condition that
$a_1,\ldots,a_n\in\mathfrak{m}$. We also define
$\tmop{Der}_{\log} (\mathcal{O}) \subset
\tmop{Der} \left( \oO \right)$ as the submodule of derivations preserving the
ideal $I_{\log} = \langle x_1 \cdots x_n\rangle\mathcal{O}$ generated by the product of all coordinates functions. It is a free module over $\mathcal{O}$ with basis $\left\{ x_i
\frac{\partial}{\partial x_i} \right\}$.

We define similarly the free modules $\tmop{Der}_{\log} \left( \tO \right)
\subset \tmop{Der} \left( \tO \right)$ of derivations and logarithmic
derivations in the ring $\tO$. Each derivation $\partial \in \tmop{Der} \left(
\oO \right)$ extends uniquely to a derivation in $\tmop{Der} \left( \tO
\right)$. We say that a derivation $\partial$ in $\tmop{Der} \left( \tO
\right)$ is {\tmem{analytic}} if $\partial$ is the extension of a derivation
in $\tmop{Der} \left( \oO \right)$.

In this paper, we consider only logarithmic derivations. To simplify the notation, we will henceforth denote the modules 
\(\tmop{Der}_{\log}(\oO)\) and \(\tmop{Der}_{\log}(\tO)\) simply by \(\cD\) and \(\tcD\), respectively.

We will frequently consider logarithmic derivations of the form
$$
x^{\um} L(\lambda),$$
where $\um \in \N^n,$ and $L (\lambda) := \sum \lambda_j x_j\frac{\partial}{\partial x_j}$ for some $\lambda \in \C^n$.
We will call them {\em monomial derivations}.  Note that a monomial derivation acts on monomials by the formula
\begin{equation}\label{Lie-deriv-monomial}
x^\um L(\lambda)\, \big(x^\un\big) = \langle \lambda , \un  \rangle\,  x^{\um+\un}\, 
\end{equation} 
and Lie bracket of two monomial derivations is given by 
\begin{equation}\label{Lie-bracket-formula} 
\big[ \, x^{\um} L(\lambda), \x^{\un} L(\mu)\, \big]  = \langle \lambda , \un  \rangle \,\x^{\um+\un} L(\mu)  - \langle \mu , \um  \rangle \x^{\um+\un}L(\lambda).
\end{equation}
In particular, $[L(\lambda), L(\mu)] = 0$. We further observe that every derivation $\partial\in\tcD$ admits a unique expansion as a sum of monomial derivations,
\begin{equation}\label{monomial-exp-der}
\partial=\sum_{\um\in\mathbb N^n}x^{\um}\,L(\lambda_{\um}),
\end{equation}
(with $\mathbb N=\mathbb Z_{\ge 0}$) where $\lambda_{\um}\in\mathbb C^n$ for every $\um\in\mathbb N^n$. We call \eqref{monomial-exp-der} the {\em monomial expansion} of~$\partial$.
Indeed, writing
\[
\partial=\sum_{i=1}^n f_i\,x_i\frac{\partial}{\partial x_i},
\]
with $f_1,\ldots,f_n\in\tO$, and expanding each coefficient as
\[
f_i=\sum_{\um\in\mathbb N^n}a_{i,\um}x^{\um},
\]
we obtain
\[
\partial=\sum_{i=1}^n\sum_{\um\in\mathbb N^n}
a_{i,\um}x^{\um}x_i\frac{\partial}{\partial x_i}
=\sum_{\um\in\mathbb N^n}
x^{\um}
\left(
\sum_{i=1}^n
a_{i,\um}\,
x_i\frac{\partial}{\partial x_i}
\right).
\]
Hence, \eqref{monomial-exp-der} is obtained by setting $\lambda_{\um}=(a_{1,\um},\ldots,a_{n,\um})\in\mathbb{C}^n$.

Based on this expansion, define the {\em support of $\partial$} by 
$$\supp(\partial) = \{ {\um}:\lambda_{\um} \ne 0\},$$ 
and define the order and degree of $\partial$ by
\begin{equation}\label{def-degree-order-deriv}
\ord(\partial) = \inf\{\|\um\|: \um \in \supp(\partial)\}\quad \text{ and }\quad \deg(\partial) = \sup\{\|\um\|: \um \in \supp(\partial)\}.
\end{equation}
We say that $\partial$ is {\em $k$-flat} if $\ord(\partial) \geq k$. Note that this is equivalent to say that 
$$
\partial(\mathfrak{m})\subset \mathfrak{m}^{k+1}.
$$
In analogy to the case of germs, we say that $\partial$ {\em polynomial} if $\deg(\partial)<\infty$, or equivalently, if $\supp(\partial)$ is finite. 

It will be also convenient to consider expansions of derivations with respect to a fixed set of diagonal derivations. A collection of diagonal derivations $L(\mu_0),\ldots,L(\mu_{n-1})$ will be called a {\em logarithmic basis} if the vectors $\mu_0,\ldots,\mu_{n-1}\in \C^n$ are linearly independent. In this case, each derivation 
$\partial \in \cD$ has a unique expansion
\begin{equation}\label{logarithmic-expansion-of-partial}
\partial = \sum_{i=0}^{n-1} g_i \, L(\mu_i)
\end{equation}
with coefficients $g_0,\ldots,g_n \in \tO$. We call it the {\em logarithmic expansion of $\partial$} with respect to the basis $\{L(\mu_i)\}$.
\begin{remark}\label{passage-between-expasions}
The transition between the monomial expansion~\eqref{monomial-exp-der} and the above expansion can be easily obtained by expressing each diagonal derivation $L(\lambda_{\underline m})$ in the basis $\{L(\mu_0),\ldots,L(\mu_{n-1})\}$ and then regrouping the terms. 
In other words, it suffices to consider the $n\times n$ matrix $\mathbf{T}$ of change of basis between the canonical basis $\{e_1,..,e_n\}$ and the basis 
$\{\mu_0,\ldots,\mu_{n-1}\}$ of $\C^n$.
\end{remark}
We shall frequently use the following fact (see e.g. \cite{Martinet1981Norm}, section 2.1): A diagonal derivation $S= L(\lambda)$ induces a graduation
		\begin{equation} \label{graduation-S}
		\tO = \bigoplus_{\alpha \in \C} \Gr_{\alpha}(\tO, S),
	\end{equation}
	where $\Gr_{\alpha}(\tO, S)$ is the linear subspace of series $f\in \tO$ such that $S(f)=\alpha f$.  We also observe that $S$ induces a graduation on the formal derivations, 
    \begin{equation} \label{graduation-S-der}
		\tcD = \bigoplus_{\alpha \in \C} \Gr_{\alpha}(\tcD, S),
	\end{equation}
    where $\Gr_{\alpha}(\tcD, S)= \{\delta \in \tcD : [S,\delta]=\alpha\, \delta\}$, where $[\, ,\, ]$ denotes the Lie bracket. 
    
    More concretely, $\Gr_{\alpha}(\tO, S)$ and $\Gr_{\alpha}(\tcD, S)$ are respectively the spaces of series and derivations with expansion of the form
    $$
    f = \sum_{\underline{m}} a_{\um} x^{\underline{m}},\qquad \partial = \sum_{\um} \x^{\um} \, L(\lambda_{\um}),
    $$
    where the sums are taken over all indices $\underline{m} \in \mathbb{N}^n$ satisfying $\langle \lambda,\underline{m} \rangle = \alpha$. We also recall the basic graduated algebra relations
    \begin{equation}\label{graduated-ids}
    \Gr_{\alpha}(\tO, S)\cdot\Gr_{\beta}(\tO, S)\subset \Gr_{\alpha+\beta}(\tO, S),\qquad  \big[\Gr_{\alpha}(\tcD, S), \Gr_{\beta}(\tcD, S)\big]\subset \Gr_{\alpha+\beta}(\tcD, S)
    \end{equation}
    Also, 
    $$\Gr_{\alpha}(\tO, S)\cdot \Gr_{\beta}(\tcD, S) \subset \Gr_{\alpha+\beta}(\tcD, S)$$
    for all $\alpha,\beta \in \C$. Similar graduations are defined over the analytic ring $\vO$ and analytic derivations $\cD$.

We conclude by recalling some facts related to derivations and the exponential map, and refer the reader to \cite{Bourbaki2007}, \textsection 6 for the detailed proofs. If $\partial$ is a local derivation of $\oO$ then its {\em exponential}
\begin{equation}\label{exponential-formula}
\Phi = \exp(\partial) = \sum_{k=0}^\infty \frac{\partial^k}{k !} = \mathbf{Id} + \partial + \frac{\partial^2}{2}+\cdots
\end{equation}
is an automorphism of $\oO$. Moreover, since we will always assume that $\partial$ is logarithmic, we can write
$$\Phi(x_i) = x_i\, u_i,\quad i=1,\ldots,n$$
for some units $u_i \in \oO$. In particular, we can consider {\em conjugation morphism} acting on $\cD$, defined by either one of the following equalities 
\begin{equation}\label{conjugation-morphism}
\delta \mapsto  \Phi\, \delta\, \Phi^{-1} = \exp(\ad_\partial) \delta = \sum_{n=0}^\infty \frac{1}{n!} \ad_\partial^n\, (\delta)
\end{equation}
where $\ad_\partial (\delta) = [\delta,\partial]$ and $\ad_\partial^n = \ad_\partial \circ \cdots \circ \ad_\partial$. 

All these properties also hold for formal derivations.

\subsection{Jordan decomposition}\label{subsect-jordan-decomp}

Let $\mathcal{J}^k = \vO / \mathfrak{m}^{k + 1}$ denote the ring of $k$-jets,
which we identify, as a vector space, with the polynomials of degree at most
$k$. 

A local derivation is called \emph{semi-simple/nilpotent} if the quotient derivation
${\partial_{\mathcal{J}^k}}:= \partial \mod \mathfrak{m}^{k+1} $ is a semi-simple/nilpotent endomorphism of $\mathcal{J}^k$, for
each $k \in \mathbb{N}$.  

By the standard Jordan--Chevalley decomposition, for any local derivation $\partial$ and any integer $k \ge 0$, the induced derivation $\partial_k = {\partial_{\mathcal{J}^k}}$ on the $k$-jet space admits a unique decomposition
$$
\partial_k = (\partial_k)_{\tmop{ss}} + (\partial_k)_{\tmop{nilp}}
$$
where $(\partial_k)_{\tmop{ss}}$ and $(\partial_k)_{\tmop{nilp}}$ are, respectively, semi-simple and nilpotent derivations on $\mathcal{J}^k$ that commute (see, e.g., \cite{humphreys1972}, Section 4.2, Lemma B). 

Furthermore, for each $0 \le l \le k$, let $\pi_{kl}: \mathcal{J}^k \rightarrow \mathcal{J}^l$ denote the canonical truncation morphism. Since the Jordan decomposition is functorial with respect to surjective ring morphisms, we have 
$$
(\partial_l)_{\tmop{ss}} = (\pi_{kl})_* (\partial_k)_{\tmop{ss}} \quad \text{and} \quad (\partial_l)_{\tmop{nilp}} = (\pi_{kl})_* (\partial_k)_{\tmop{nilp}}.
$$
i.e.~the Jordan decomposition is compatible with truncations. Therefore, taking the projective limit as $k \rightarrow \infty$, we conclude that each local derivation $\partial \in \tcD$ admits a unique {\tmem{Jordan decomposition}} $\partial = \partial_{\tmop{ss}} + \partial_{\tmop{nilp}}$, where $\partial_{\tmop{ss}}, \partial_{\tmop{nilp}} \in \tcD$ are respectively semi-simple and nilpotent derivations satisfying 
$[\partial_{\tmop{ss}}, \partial_{\tmop{nilp}}] = 0$ and 
$$
(\partial_{\tmop{ss}})_{\mathcal{J}^k}=(\partial_k)_{\tmop{ss}},  \qquad(\partial_{\tmop{nilp}})_{\mathcal{J}^k}=(\partial_k)_{\tmop{nilp}} 
$$
for each $k\ge 0$. 
\begin{remark} We observe the following facts:
  \label{remark-invJordan} 
  \begin{enumerate}
  \item The definition of a nilpotent local derivation is formulated at the level of jets and can be equivalently stated as follows: a derivation $\partial$ on a local ring $(R, \mathfrak{m})$ is nilpotent if for every integer $k \ge 0$, there exists an integer $n = n(k)$ such that $\partial^n(\mathfrak{m}) \subset \mathfrak{m}^k$. For instance, the derivation $\partial = x^2 \frac{\partial}{\partial x}$ on $\mathbb C[[x]]$ is nilpotent: since $\partial^n(x) = n! x^{n+1}$. 
  \item The assumption that $\partial \in \tcD$ implies that $\partial_{\tmop{ss}}$ and $\partial_{\tmop{nilp}}$ also belong to $\tcD$.   
  \item An ideal $J \subset \widetilde{\mathcal{O}}$ is
  $\partial$-invariant if and only if it is both $\partial_{\tmop{ss}}
  $ and $\partial_{\tmop{nilp}}$ invariant.
  \item If $J$ is a $\partial$-invariant ideal, one can define the relative Jordan decomposition of the induced derivation $\partial_A$ on the quotient algebra $A=\tO/J$.
The construction is exactly the same as above, but carried out on the sequence of relative jet spaces
$\mathcal{J}^k_A = A/(\mathfrak{m}^{k+1}/J)$.
It follows from item (2) and from the uniqueness of the Jordan decomposition that this relative decomposition is given by
$$\partial_A = (\partial_{\mathrm{ss}})_A + (\partial_{\mathrm{nilp}})_A.$$
In fact, one can prove that the semisimple–nilpotent decomposition is {\em functorial}, meaning that it is preserved under arbitrary (formal) ring morphisms. 
\end{enumerate}
\end{remark}

We observe that, from the definitions, it follows that any semi-simple derivation is formally diagonalizable over $\mathbb{C}$ (see \cite{Martinet1981Norm}, section 2). More precisely, eventually extending the base field from $\mathbb{R}$ to $\mathbb{C}$, there exists a formal automorphism $\Phi$ of $\tO$ such that we can write
$$
S := \Phi\, \partial_{\tmop{ss}}\, \Phi^{-1} = \sum_{i=1}^n \lambda_ix_i \frac{\partial}{\partial x_i}.
$$ 
In particular, any derivation $\partial \in \tcD$ is formally conjugated to a Poincar\' e-Dulac normal form, according to the definition given in the Introduction. 

Moreover, it follows from (\ref{graduation-S}) that $\partial_{\tmop{ss}}$ defines a graduation
\begin{equation}\label{graduation-partialss}
\tO = \bigoplus_{\alpha \in \C} \Gr_{\alpha}(\tO, \partial_{\tmop{ss}}),
\end{equation}
where the degree $\alpha$-component, $\Gr_{\alpha}(\mathcal O, \partial_{\tmop{ss}})$ is the set of series $f \in \tO$ such that $\partial_{\tmop{ss}}(f) = \alpha f$. Note that we can also write 
$$\Gr_{\alpha}(\mathcal O, \partial_{\tmop{ss}}) = \Phi\big( \Gr_{\alpha}(\mathcal O, S)\, \big),$$
i.e.~$\Gr_{\alpha}(\mathcal O, \partial_{\tmop{ss}})$ is the image of the subspace $\Gr_{\alpha}(\mathcal O, S)$ under the diagonalizing automorphism $\Phi$ defined above.

It is important to remark that, even if $\partial$ is analytic, its semi-simple and nilpotent components $\partial_{\tmop{ss}}$ and $\partial_{\tmop{nilp}}$ are not necessarily analytic, as the following example illustrates.
\begin{example}
    Consider a germ of analytic {\em saddle-node singularity} $\partial= y\yy + f(x,y)x \xx$ with $f\in \mathfrak{m}$ and suppose that $\partial_{\tmop{ss}}$ and $\partial_{\tmop{nilp}}$ are both analytic. Since $\partial_{\tmop{ss}}$ is formally linearizable (that is, formally conjugate to $y\yy$) and clearly satisfies the $\omega$-condition, Theorem~\ref{thm:Bruno-normalization} implies that $\partial_{\tmop{ss}}$ is analytically conjugate to $y\yy$. Therefore, up to analytic conjugacy, we can assume that $[\partial_{\tmop{nilp}}, y\yy]=0$, which implies that we can write $\partial_{\tmop{nilp}}= F(x) x \xx + G(x) y \yy$, with $F,G\in \vO$ vanishing at $0$; and the factorization
 $$
 \partial = (1+G(x))\, y\yy + F(x) \xx = u(x) \left( y\yy + H(x)\, x\xx\right),
 $$
with $u=1+G$ a unit and $H=F/u$. Then, by a further analytic coordinate change involving only the $x$-variable, we can assume that $\partial$ has the form
    $$  \partial= v(x)\left( y\yy + \frac{x^{k+1}}{1+\lambda x^k} \xx\right).$$
with some new analytic unit $v$. However, this special form implies the holonomy of $\partial$ along the strong separatrix $x=0$ can be embedded into a flow. This is a very restrictive condition.  Indeed, for a {\em generic} choice of $f$ one can show that the holonomy of $\partial$ along the strong separatrix cannot be embedded into a flow (see e.g.~\cite{Martinet-Ramis1983}). We refer the reader to \cite{Elizarov1985} for a systematic discussion of a large classes of examples presenting similar phenomena.
\end{example}
{\bf Assumption:} {\em From now on, we shall tacitly make the assumption that the derivation $\partial$ satisfies the {\em non-degeneracy condition} }
$$\partial_{\tmop{ss}} \ne 0.$$
Note that this assumption merely excludes the trivial case of our main theorem. Indeed, if $\partial_{\tmop{ss}} = 0$, then $\partial = \partial_{\tmop{nilp}}$ is already in normal form and there is nothing further to prove.\subsection{Collinearity and Bruno ideal}\label{def-collinearity-id}
We recall that the {\tmem{ideal of collinearity}} of two derivations $\partial_1,
  \partial_2 \in \tcD$ is the ideal 
  $$\Gamma (\partial_1 \wedge \partial_2) \subset
  \tO,$$
  as defined in subsection \ref{sec:prelim}. We observe that if $\partial_1,\partial_2$ are analytic then $\Gamma (\partial_1 \wedge \partial_2)$ is analytic. Indeed, if we write the expansions
$$
\partial_1 = \sum_{i=1}^n f_{1,i}\, x_i \frac{\partial}{\partial x_i} \quad \text{and} \quad \partial_2 = \sum_{i=1}^n f_{2,i} \, x_i \frac{\partial}{\partial x_i},
$$
where $f_{1,i}$ and $f_{2,i}$ are germs in $\vO$, then the ideal $\Gamma (\partial_1 \wedge \partial_2)$ is generated by the $2 \times 2$ minors $f_{1,i} f_{2,j} - f_{1,j} f_{2,i}$ for all $1 \le i < j \le n$.
  
More generally, suppose that $\partial_1, \partial_2$ preserve some ideal $I \subset \tO$
  (i.e.~that $\partial_i(I)\subset I$ for $i=1,2$). Then, we define the
  {\tmem{mod-$I$ ideal of collinearity}} by the ideal sum
   \[ \Gamma (\partial_1 \wedge \partial_2) + I. \]
  \begin{remark}
  Let us suppose that $I$ is an analytic ideal and that $\partial$ is an analytic derivation. Then
  the condition of invariance $\partial (I) \subset I$ is equivalent to saying
  that the germ of analytic variety $V(I)$ has the following property:\ \
  \[ \forall p \in V (I) : \partial_p \in T_p V  (I), \]
  where $\partial_p$ denote the vector defined by $\partial$ in the tangent
  space $T_p \mathbb{C}^n$ at $p$ and $T_p V (I) $ is the Zariski tangent space of $V
  (I)$ at $p$.
  
  Let us now assume that $\partial_1, \partial_2$ are analytic derivations satisfying $\partial_1 (I), \partial_2 (I) \subset I$ and consider the mod-$I$ ideal of collinearity $J
  = \Gamma (\partial_1 \wedge \partial_2) + I$. Then the subvariety $V (J)
  \subset V (I)$ is defined by
  \[ V (J) = \{ p \in V (I) : \partial_{1, p} \wedge \partial_{2, p} = 0 \}, \]
 i.e. the set of points in $V(I)$ where the $2$-vector $\partial_1 \wedge \partial_2$ vanishes.
 \end{remark}
 The following ideal plays a prominent role in the present work.
    \begin{definition}
  The {\tmem{logarithmic Bruno ideal}}, or simply, the {\em Bruno ideal} of a derivation $\partial \in \tcD$
  is the formal ideal defined by
  \[ B (\partial) = \Gamma (\partial_{\tmop{ss}} \wedge
     \partial_{\tmop{nilp}}) .\]
\end{definition}
\noindent In other words, $B (\partial)$ is the ideal of collinearity of the semi-simple and
  nilpotent components of $\partial$.
\begin{remark}
As mentioned in the introduction, the definition of the ideal $B(\partial)$ given above is inspired by the ideals $\mathcal{B}(\partial)$ and $\mathcal{A}(\partial) \subset \mathcal{B}(\partial)$ originally introduced by Bruno (see, e.g., \cite{Bruno1989}, p.~297), although they do not coincide in general. We plan to explore the precise relationship between these ideals in future work.
\end{remark}
\begin{lemma}
  \label{lemma-(1+m)}Let $\partial \in \tcD$. Then, $B (\partial)$ is $\partial$-invariant and, in the
  quotient algebra $A = \tO / B (\partial)$, we can write
  \[ \partial_A = (1 + h_A) (\partial_{\tmop{ss}})_A, \]
  where $h$ is a element of $\tmop{Gr}_0 \left( \tO, \partial_{\tmop{ss}}
  \right) \cap \mathfrak{m}$.
\end{lemma}
\noindent We recall that $\tmop{Gr}_0 \left( \tO, \partial_{\tmop{ss}}
  \right)$ denotes the degree-0 component of the graduation given in (\ref{graduation-partialss}).
\begin{proof} 
  Using the non-degeneracy assumption, we complete $\partial_{\tmop{ss}}$ to a commuting logarithmic basis $\{
  \partial_{\tmop{ss}}, L_1, \ldots, L_{n - 1} \}$ of $\tcD$. Then we can
  write the expansion
  \[ \partial = \partial_{\tmop{ss}} + \partial_{\tmop{nilp}} =
     \partial_{\tmop{ss}} + h \partial_{\tmop{ss}} + \sum_{k = 1}^{n - 1} b_k
     L_k, \]
  where we necessarily have $h, b_1, \ldots, b_{n - 1} \in \mathfrak{m}$.  Indeed, assume for the sake of contradiction that at least one coefficient among $h, b_1, \dots, b_{n - 1}$ lies in $\tO \setminus \mathfrak{m}$ (i.e., is a unit). Rewriting the expansion of $\partial_{\tmop{nilp}}$ in terms of the standard logarithmic basis $\left\{x_i \frac{\partial}{\partial x_i}\right\}$, we have
$$
\partial_{\tmop{nilp}} = h \partial_{\tmop{ss}} + \sum_{k = 1}^{n - 1} b_k L_k = \sum_{i = 1}^{n} c_i x_i \frac{\partial}{\partial x_i},
$$
which implies that at least one coefficient, say $c_j$, must also be a unit. Therefore, applying $\partial_{\tmop{nilp}}$ to $x_j$, we obtain
$$
\partial_{\tmop{nilp}}(x_j) = c_j x_j = (\text{unit}) \cdot x_j.
$$
Consequently, by the Leibniz rule, $(\partial_{\tmop{nilp}})^k(x_j) = (\text{unit}) \cdot x_j$ for all $k \ge 1$. However, according to Remark~\ref{remark-invJordan}(1), this contradicts the assumption that $\partial_{\tmop{nilp}}$ is a nilpotent derivation.
  
 The commutativity
  relation $[\partial_{\tmop{ss}}, \partial_{\tmop{nilp}}] = 0$ implies also that
  $h, b_1, \ldots, b_{n - 1} \in \tmop{Gr}_0 \left( \tO,
  \partial_{\tmop{ss}} \right)$. Moreover, the definition of the ideal $B (\partial)$
  implies that it is generated by $\langle b_1, \ldots, b_{n - 1} \rangle$. An
  easy computation shows that $B (\partial)$ is both $\partial_{\tmop{ss}}$
  and $\partial_{\tmop{nilp}}$ invariant. Hence, it is $\partial$-invariant.
  The last statement is obvious.
\end{proof}

\begin{remark}\label{remark-Bnotnecanalytic}
  As remarked above, even if we assume that $\partial$ is an analytic
  derivation, we cannot in general guarantee that $B (\partial)$ is an
  analytic ideal since the semi-simple and nilpotent components of $\partial$ are not necessarily analytic.
\end{remark}
For an ideal $J \subset \tilde{\mathcal{O}}$, we denote by
  $\Phi^{\ast} J$ the ideal generated by $\{f \circ \phi \mid f \in J\}$, with
  $\phi$ being the coordinate change associated with $\Phi$. In other words, $\phi$ is the formal substitution of variables $(x_1,...,x_n) \mapsto (\phi_1(x),..,\phi_n(x))$ where $\phi_1,...,\phi_n \in \tO$ are series defined by
 $$
 \phi_1=\Phi(x_1),\ldots,\phi_n=\Phi(x_n).
 $$
The following result shows the intrinsic nature of the Bruno ideal. 
\begin{lemma}
  \label{lemma-pull-back-B} For any automorphism $\Phi$ of
  $\tilde{\mathcal{O}}$, one has
  \[ B (\Phi \partial \Phi^{- 1}) = \Phi^{\ast} B (\partial). \]
  \end{lemma}

  \begin{proof}
      
  Let $\delta = \Phi \partial \Phi^{- 1}$. Since the Jordan
  decomposition is preserved under conjugation, we have
 $\delta_{\tmop{ss}} = \Phi \partial_{\tmop{ss}} \Phi^{- 1}$ and $\delta_{\tmop{nilp}}=
     \Phi \partial_{\tmop{nilp}} \Phi^{- 1}$. Therefore, computing in $ \widetilde{\DD} \wedge \widetilde{\DD}$, we obtain the identity
\[\delta_{\tmop{ss}} \wedge
     \delta_{\tmop{nilp}} = \Phi (\partial_{\tmop{ss}} \wedge
     \partial_{\tmop{nilp}}) \Phi^{- 1} . \]
  If we choose an arbitrary free basis $(B_i \wedge B_j)_{i < j}$ of $\widetilde{\DD}
  \wedge \widetilde{\DD}$, and write
  \[ \partial_{\tmop{ss}} \wedge \partial_{\tmop{nilp}} = \sum_{i < j} a_{ij} 
     \hspace{0.17em} B_i \wedge B_j, \]
  then the above identity gives
  \[ \delta_{\tmop{ss}} \wedge \delta_{\tmop{nilp}} = \Phi
     (\partial_{\tmop{ss}} \wedge \partial_{\tmop{nilp}}) \Phi^{- 1} = \sum_{i
     < j} (\Phi^{\ast} a_{ij})  \hspace{0.17em} (\Phi B_i \Phi^{- 1}) \wedge
     (\Phi B_j \Phi^{- 1}) . \]
  Hence, as $B (\partial)$ is the ideal generated by the coefficients
  $\{a_{ij} \}_{i < j}$, the ideal $B (\delta)$ is generated by the transformed
  coefficients $\{\Phi^{\ast} a_{ij} \}_{i < j}$.\footnote{This follows from
  the natural action of $\Phi$ on $\widetilde{\DD} \wedge \widetilde{\DD}$.}
\end{proof}
     Note that if we consider the associated quotient algebras
  \[ A = \widetilde{\mathcal{O}} / B (\partial)  \quad \text{and} \quad A' =
     \widetilde{\mathcal{O}} / B (\Phi \partial \Phi^{- 1}),  \]
  then the above lemma implies that the automorphism $\Phi$ induces a natural algebra morphism (which we denote by the same letter) $\Phi : A \to
  A'$, and the induced derivations satisfy
  \[ (\Phi \partial \Phi^{- 1})_{A'} = \Phi \hspace{0.17em} \partial_A 
     \hspace{0.17em} \Phi^{- 1} . \]
The following lemma will play an important role in the proof of the Proposition \ref{prop:formalinductivestep}. 
\begin{lemma}
  \label{Lemma-goestoquotient}Suppose that $I$ is a {$\partial$}-invariant
  ideal and that there exists two derivations $S$, $R \in \tcD$ satisfying the
  following conditions:
  \begin{enumeratealpha}
    \item $S$ is semi-simple, $R$ is nilpotent and $[S, R] = 0$.
    
    \item The ideal $I$ is $(S + R)$-invariant and, in the quotient algebra $A
    = \tO / I$ we have the equality
    \[ \partial_A = (S + R)_A. \]
  \end{enumeratealpha}
  Then the ideal $I$ is both $S$-invariant and $R$-invariant. Moreover
  \begin{equation}
    (\partial_{\tmop{ss}})_A = S_A, \quad (\partial_{\tmop{nilp}})_A = R_A.
    \label{condition-A}
  \end{equation}
\end{lemma}

\begin{proof}
  Consider the derivation $\delta : = S + R$, which has Jordan decomposition
  $\delta_{\tmop{ss}} = S$ and $\delta_{\tmop{nilp}} = R$. The item $(b)$ of the enunciate
  implies that $\partial - \delta \in I \tcD$, by property~$(D.1)$ in Section~\ref{subsec:derivfreemodule}. Note also that $I$ is
  $\delta$-invariant. Hence, by the Remark \ref{remark-invJordan}(3), $I$ is both
  $S$ and $R$ invariant. The final statement is obvious. 
\end{proof}
Using a very similar argument, we can establish the following {\em relative version} of the above lemma:
\begin{lemma}\label{Lemma-relative-version}
Suppose that $J \subset I$ are ideals in $\tO$ such that $J$ is invariant under $\partial$, $S$, and $R$. Let $S_C$ and $R_C$ denote the derivations induced by $S$ and $R$ on the quotient algebra $C = \tO / J$, and let $K = I/J$. Suppose further that the following conditions hold:
\begin{enumeratealpha}
  \item $S_C$ is semi-simple, $R_C$ is nilpotent, and $[S_C, R_C] = 0$.
  \item The ideal $K$ is invariant under $S_C + R_C$ and, on the quotient algebra $A = C/K$, we have the equality
  \[ \partial_A = (S + R)_A. \]
  \item The ideal $I$ is invariant under $\partial$.
\end{enumeratealpha}
Then the ideal $K$ is invariant under $S_C$ and $R_C$, and property~\eqref{condition-A} holds in the quotient algebra $A$ (which is isomorphic to $\tO / I$ by the Third Isomorphism Theorem).
\end{lemma}
\begin{proof}
Recall the derivation $\delta = S + R$. Since $J$ is invariant under both $S$ and $R$, the derivations $\delta$, $S$, and $R$ induce derivations on the quotient algebra $C = \tO / J$, which we denote by $\delta_C = S_C + R_C$. Furthermore, by hypothesis we have
$$
(\delta_C)_{\tmop{ss}} = S_C \quad \text{and} \quad (\delta_C)_{\tmop{nilp}} = R_C;
$$
that is, $S_C$ and $R_C$ are, respectively, the semi-simple and nilpotent components of the Jordan decomposition of $\delta_C$ (see Remark~\ref{remark-invJordan}(4)).

By assumption, the ideal $K = I/J$ of $C$ is invariant under $\delta_C$. The general property of the Jordan decomposition implies that $K$ is also invariant under $S_C$ and $R_C$. Furthermore, hypothesis (b) states that $K$ is invariant under both $\delta_C$ and $\partial_C$. Therefore, the difference $\delta_C - \partial_C$ belongs to the ideal $K \tcD_C$, where $\tcD_C$ denotes the module of logarithmic derivations on $C$. Passing to the quotient ring $A = C/K$, this implies that the induced derivations coincide ($\delta_A = \partial_A$). Consequently, we obtain
$$
(\partial_{\tmop{ss}})_A = S_A \quad \text{and} \quad (\partial_{\tmop{nilp}})_A = R_A.
$$
\end{proof}
\subsection{Normal form and normal form modulo an ideal}\label{subsect-normalformsdef}
From now on, we fix a non-zero diagonal derivation $S=L(\lambda)$.
\begin{definition}\label{def-S-perturb}
An \tmem{$S$-perturbation} is a formal derivation
$\partial \in \tcD$ with a monomial expansion of the form
\begin{equation}\label{mon-expansion-S-pert}
  \partial = S + \sum_{\|\um\| \ge 1} x^{\um} \, L(\lambda_{\um}).
\end{equation}
\end{definition}
As in the introduction, we say that $\partial$ is in {\tmem{normal form}} if 
$\partial_{\tmop{ss}}=S$, or, equivalently, if $$\supp(\partial) \subset \{\um : \langle \lambda, \um \rangle = 0\}.$$ 
We now introduce a {\em relative} version of such concept, by considering the restriction of $\partial$ to a quotient algebra by an invariant ideal. 
\begin{definition}\label{def-nf-invideal}
Suppose that $I \subset \widetilde{\vO}$ is a
$\partial$-invariant ideal. We say that an $S$-perturbation $\partial$ is in {\tmem{normal form
modulo $I$}} if the following two conditions hold:
\begin{itemizedot}
  \item The ideal $I$ is $S$-invariant and
  \item $(\partial_{\tmop{ss}})_A = S_A$ in the quotient algebra $A = \tO /
  I$.
\end{itemizedot}
\end{definition}
Note that, when $I=0$, we recover the usual notion of normal form. In the next enunciate, we suppose that $\partial$ is given by (\ref{mon-expansion-S-pert}) and consider the nilpotent derivation 
$$R = \partial - S = \sum_{\|\um\| \ge 1} x^{\um} \, L(\lambda_{\um}).
$$
\begin{proposition}\label{prop:normalform-SR-invariant}
  If $\partial$ is in normal form modulo $I$ then:
  \begin{enumerate}
    \item $I$ is $R$-invariant and
    
    \item $(\partial_{\tmop{nilp}}) _A = R_A$
  \end{enumerate}
  As a consequence, $[S_A, R_A] = 0$ (i.e. $S$ and $R$ commute modulo $I$). 
\end{proposition}
\begin{proof}
  The first statement is clear. Furthermore, in the quotient algebra we have
  $R_A = (\partial - S)_A = \partial_A - S_A = \partial_A -
  (\partial_{\tmop{ss}})_A = (\partial_{\tmop{nilp}})_A$.
\end{proof}
Let us observe that the relation $[S_A, R_A] = 0$ is equivalent to state that
\
\[ \left[ S, \sum_{\|\um\| \ge 1} x^{\um} \, L(\lambda_{\um}) \right] \in I \tcD. \]
Applying (\ref{Lie-bracket-formula}), this is also equivalent to say that $S (x^\um) \in I$ for all $\um \in \supp(\partial)$.
\begin{remark}\label{remark-S(f)}
For future reference, let us reformulate the above computation. Fix an arbitrary collection of diagonal vector fields $T_1=L(\mu_1),\ldots,T_{n-1}=L(\mu_{n-1})$ such that $\{S, T_1, \dots, T_{n-1}\}$ forms a logarithmic basis. As in Subsection~\ref{subsect-ideals-expansions}, we consider the logarithmic expansion
\begin{equation}\label{logarithmic-expansion-here}
\partial = S + R = S + \left( f S + \sum_{j=1}^{n-1} g_j T_j\right),
\end{equation}
where $f,g_j \in \mathfrak{m}$. Since $\{S, T_1, \dots, T_{n-1}\}$ constitutes a basis, the condition that $S(x^\um) \in I$ for all $\um \in \supp(\partial)$ is equivalently formulated as
$$
S(f), S(g_1), \dots, S(g_{n-1}) \in I.
$$
\end{remark}
We shall mainly be interested in the case where the invariant ideal $I$ in the definition \ref{def-nf-invideal} is precisely the Bruno ideal $B(\partial)$ of $\partial$. Let us show that in this case the induced derivation $\partial_A$ in the quotient algebra 
\begin{equation}\label{algebraA}
A=\tO/B(\partial)
\end{equation}
assumes a very special form.
\begin{corollary}\label{cor:coeffsinB}
  \label{Corollary-coeffinB}Suppose that $\partial$ is in normal form modulo
  $B (\partial)$, where $B (\partial)$ is the Bruno ideal of $\partial$. Then, referring to the expansion (\ref{logarithmic-expansion-here}), 
  we can write
  \[ \partial_A = (1 + f_0) S_A, \]
  where $A$ is given by (\ref{algebraA}) and $f_0 \in \Gr_0(\tO, S)$ is the degree 0 component of $f$ with respect to the graduation (\ref{graduation-S}). In particular, the elements
  \[ f - f_0, g_1, \ldots, g_{n - 1} \]
  belong to the Bruno ideal. 
\end{corollary}

\begin{proof}
  By the Lemma \ref{lemma-(1+m)}, we have
  \[ \partial_A = (1 + h) (\partial_{\tmop{ss}})_A = (1 + h) S_A, \]
  where $h$ is some element of $\tmop{Gr}_0 \left( \tO, \partial_{\tmop{ss}}
  \right) \cap \mathfrak{m}$.  This shows that
  \[ \sum g_j T_j \in B (\partial) \tcD. \]
  Since $\{ S, T_1, \ldots, T_{n-1} \}$ \ forms a basis of $\tcD$, we conclude that $g_j \in B
  (\partial)$. Moreover, $1 + h$ is congruent to $1 + f$ modulo $B
  (\partial)$. As we remarked above $S (f) = S (f - f_0) \in B (\partial)$,
  which implies that $f - f_0 \in B (\partial)$ (\footnote{Since the ideal $B(\partial)$ is $S$-invariant, it is also an $S$-graded ideal of $\tO$. Therefore, an element $h \in \tO$ belongs to $B(\partial)$ if and only if all its $S$-homogeneous components belong to $B(\partial)$. In other words, given the $S$-graded expansion
\[
h = \sum_{\alpha} h_{\alpha},
\]
we have $h \in B(\partial)$ if and only if $h_{\alpha} \in B(\partial)$ for each degree $\alpha$. Applying the derivation $S$ to this expansion yields
\[
S(h) = \sum_{\alpha} \alpha h_{\alpha}.
\]
Because $B(\partial)$ is a graded ideal, $S(h) \in B(\partial)$ if and only if its homogeneous components satisfy $\alpha h_{\alpha} \in B(\partial)$ for all $\alpha$. This holds if and only if $h_{\alpha} \in B(\partial)$ for all $\alpha \neq 0$, which is equivalently stated as $h - h_0 \in B(\partial)$.}). Therefore, $((1 + f) S)_A = (1 +
  f_0) S_A$.
\end{proof}
\subsection{Analyticity of the Bruno ideal and Corollary \ref{cor:lienarfoliation}}
Using the results of the previous subsection, we will show that the first statement of the Main Theorem is in fact consequence of the second one. 

More precisely, for an {\em analytic} derivation $\partial$, the analyticity of the Bruno ideal $B(\partial)$ holds if we assume that  $\partial$ is in normal form modulo $B(\partial)$ (see Definition \ref{def-nf-invideal}). 

As in the previous subsection, we assume that $\partial$ is an $S$-perturbation (Definition \ref{def-S-perturb}). However, we now require that the second term in the expansion \eqref{mon-expansion-S-pert},
$$R = \sum_{\Vert{}\um\Vert{} \ge 1} x^{\um} \, L(\lambda_{\um}),$$
is an \emph{analytic} derivation satisfying $\ord(R) \ge 1$. Hence, since $S=L(\lambda)$, the derivation $\partial$ is a also an analytic.

We also fix an arbitrary logarithmic expansion as in \eqref{logarithmic-expansion-here} where now $f,g_j$ are analytic germs. Based on this expansion, we define the following analytic ideals, using the notation introduced in subsection \ref{subsec:derivfreemodule}:
\begin{itemize}\setlength{\itemsep}{0.4em}
  \item[(1)]The {\tmem{ideal of collinearity}} $I_1 := \Gamma (S \wedge R)$, which is generated
  by $g_1, \ldots, g_{n - 1}$.
  \item[(2)] The {\tmem{ideal of commutativity}} $I_2 := \Gamma ([S, R])$, which is generated
  by $S (f), S (g_1), \ldots, S (g_{n - 1}) .$
  \item[(3)] The ideal sum $W = I_1 + I_2$
 \item[(4)] The differential closure $I = S [W]$ of $W$ with respect to the derivation $S$ (in the ring $\OO$). 
\end{itemize}
where we recall that, in a ring $R$, the {\em differential closure} of an ideal $J\subset R$ with respect to a derivation $\delta$ is the smallest ideal $\delta[J]\subset R$ containing $I$ which is stable by $\delta$. 
 \begin{proposition}\label{prop-analyticity-Bruno-ideal}
  Suppose that $\partial$ is in normal form modulo $B(\partial)$. Then $B
  (\partial) = I \tO$. In particular $B (\partial)$ is an analytic ideal.
\end{proposition}
\noindent \begin{remark}
We note that this result imposes no arithmetic condition on the eigenvalues of \(S\).
\end{remark}
\begin{proof}
  It follows from Corollary \ref{cor:coeffsinB} that $g_1, \ldots, g_{n -
  1} \in B (\partial)$. From the Remark \ref{remark-S(f)} we also have $S
  (f), S (g_1), \ldots, S (g_{n - 1}) \in B (\partial)$. Therefore, $W \subset
  B (\partial)$. Since $B (\partial)$ is invariant by $S$, the differential
  closure $I = S [W]$ of $W$, as defined in the above item (4),  is still contained in $B (\partial)$. Therefore,
  $I \tO \subset B (\partial)$.
  
  We now prove the inverse inclusion
  \[ B (\partial) \subset I \tO . \]
  Note that $I$ is both $S$-invariant and $R$-invariant and therefore also
  $\partial, \partial_{\tmop{ss}}$ and $\partial_{\tmop{nilp}}$-invariant. By
  the definition of $I$, the derivations $S_C$ and $R_C$ commute in the
  quotient algebra $C = \tO / I \tO$. Since they are respectively semi-simple
  and nilpotent, it follows from Lemma \ref{Lemma-relative-version} (with $J=I$) that \
  \[ (\partial_{\tmop{ss}})_C = S_C, \quad \tmop{and} \quad
     (\partial_{\tmop{nilp}})_C = R_C. \qquad \]
  As a consequence, $\partial_{\tmop{ss}} - S$ and $\partial_{\tmop{nilp}} -
  R$ lie in the submodule $I \tcD$ (according to the general fact (D.1) of Section~\ref{subsec:derivfreemodule}). Therefore,
   \[ I \tO = (\Gamma (S \wedge R) + I)  \tO = \Gamma (\partial_{\tmop{ss}}
     \wedge \partial_{\tmop{nilp}}) + I \tO = B (\partial) + I \tO , \]
  which shows that $B (\partial)$ is contained in $I \tO$. 
\end{proof}
We conclude this section proving Corollary~\ref{cor:lienarfoliation}, which follows immediately from the above Proposition. 
\begin{proof}
    [Proof of Corollary~\ref{cor:lienarfoliation}]
Suppose that $\partial$ is in normal form modulo $B(\partial)$ and recall that $V(B(\partial))$ is an analytic invariant variety for $\partial$. Moreover, if we set $A=\mathcal O/ B(\partial)$ then 
$$\partial_A = (1+f_0)S_A,$$ 
for some analytic germ $f_0 \in Gr_0(\OO,S)$. Hence, in restriction $V(B(\partial))$, the foliation defined by $\partial$ is linear.
\end{proof}
\section{Formal reduction to normal form modulo the Bruno ideal}\label{sec:formalNormalizationModBruno}
The Proposition \ref{prop-analyticity-Bruno-ideal} from the last subsection implies that the Main Theorem will hold if we prove that, under the $\omega$-condition, any analytic derivation is analytically conjugated to a derivation in normal form modulo its Bruno ideal. 

To prove this result, we adapt the Newton-type inductive scheme introduced by Bruno. The main idea is to construct a sequence of coordinate changes that successively transform $\partial$ into a (relative) normal form, modulo truncations to jet spaces of increasing order. As in Newton’s method, the order of truncation is doubled at each step.

In this section, we describe how the relative normalization procedure works for formal derivations, for which (formal) convergence is straightforward to establish and requires no arithmetic assumptions. In particular, we adapt the normalization scheme of Bruno to the relative normalization. We use the formalism of the Jordan decomposition as well as results about operators on algebras presented in Section~\ref{sec:prelim}. Then, the main result in this chapter is Theorem~\ref{thm:formalnormalizationtruncated}, which shows that there is a formal automorphism that conjugates any formal derivation to one in normal form modulo its Bruno ideal.
\subsection{The adjoint operator}
The inductive reduction to normal form relies on the invertibility of the following adjoint operator, acting on $\tcD$,
$$\ad_{(1+f_0) S}(\cdot) = [\,(1+f_0)\, S,\, \cdot\,], $$
where $f_0$ is a series belonging to $\mathfrak{m}\cap \Gr_0(\tO, S)$. As a first remark, we observe that such operator restricts to a linear map
$$
\Gr_\alpha(\tcD, S) \rightarrow \Gr_\alpha(\tcD, S)
$$
for each $\alpha \in \C$, as it easily follows from (\ref{graduated-ids}). In other words, $\ad_{(1+f_0) S}$ is a graduated operator of degree $0$. 
The following result will play a crucial role in the formal and analytic reduction to a normal form.
\begin{lemma}\label{lemma-formula-inverse-adjoint}
Let $f_0$ be as above and suppose that $\alpha \in \C \setminus \{0\}$. Then, $\ad_{(1+f_0) S}$ restricts to an isomorphism of $\Gr_\alpha(\tcD, S)$. More precisely, its inverse is given explicitly by 
$$
\ad_{(1+f_0)S}^{\,-1}\, x^\um L(\mu)  = \frac{x^\um}{\alpha\,(1+f_0)} \; L(\mu) + \frac{x^\um\, L(\mu)(f_0)}{\alpha^2\,(1+f_0)^2}\; S,
$$ 
where $L(\mu)$ is any diagonal derivation and $x^\um$ denotes a monomial in $\Gr_\alpha(\tO, S)$. 
\end{lemma}
\begin{proof}
This is an straightforward computation using the following decomposition 
\[
\ad_{(1+f_0)S}(U) = [(1+f_0)S,U] =  \alpha\,(1+f_0)\cdot \left( U - \frac{U(f_0)}{\alpha\,(1+f_0)} S \right),
\]
where $U \in \Gr_\alpha(\tcD, S)$. Note that the linear operator inside the parenthesis has the form $\mathbf{Id} - \mathbf{N}$ with $\mathbf{N}$ being the linear operator
$$
U \mapsto \mathbf{N}(U) = \frac{U(f_0)}{\alpha\,(1+f_0)} S,
$$
and, since $S(f_0) = 0$, we have $\mathbf{N}^2 = 0$ (i.e.~$\mathbf{N}$ is nilpotent of degree $2$). Furthermore, it commutes with the scalar operator of multiplication by $\alpha\, (1+f_0)$. Therefore, the inverse of $\ad_{(1+f_0)S}$ can be written as $\frac{1}{\alpha\,(1+f_0)}(\mathbf{Id} - \mathbf{N})$.  This is precisely the expression given in the enunciate.
\end{proof}
Let us state a simple corollary which will be used later. Given a derivation $\delta \in \tcD$, we denote by 
$$
\delta = \sum_\alpha\; \delta_\alpha
$$
its decomposition with respect to the graduation given in (\ref{graduation-S-der}), and let $\delta_\ast=\delta - \delta_0$ denote the sum of those graduated components of $\delta$ of non-zero degree.  
\begin{corollary}\label{cor-uniquesol-cohomological-eq}
(1) Given a derivation $\delta \in \tcD$, the equation
$$\big[ (1+f_0) S,U \big]=\delta_\ast$$
has a solution $U \in \tcD$. Moreover, such solution is unique if we assume that $U=U_\ast$. \\
\noindent (2) If we additionally assume that $\ord(\delta) \ge K$ for some $K \in \mathbb{N}$ (equivalently, $\delta \in \mathfrak{m}^{K} \tcD$), then the unique solution defined above also has order at least $K$. 
\end{corollary}
The second part of the enunciate is an immediate consequence of the explicit form $\ad_{(1+f_0)S}^{\,-1}$ given by Lemma~\ref{lemma-formula-inverse-adjoint}.
\subsection{The inductive formal reduction}
From now on, we will adopt the following notation. We denote by $\mathfrak{m}^{(k)} = \mathfrak{m}^{\,2^{k}}$the $2^{k}$-st power of the maximal ideal, for $k \in \mathbb{N}$, and let
$$
\mathcal{J}^{(k)} = \mathcal{O} \big/ \mathfrak{m}^{(k)} .
$$
denote the corresponding jet space. Given a derivation $\partial$, we denote by
$$\partial^{(k)}\in \tmop{Der}\left( \tO \right)$$  
its {\em $\mathfrak{m}^{(k)}$-truncation}, i.e.~the derivation on the quotient algebra $\mathcal{J}^{(k)}$. Also, we define the {\em $\mathfrak{m}^{(k)}$-truncated Bruno ideal of $\partial$ by}
$$B(\partial^{(k)})= \Gamma (\partial_{\tmop{ss}}^{(k)} \wedge
     \partial_{\tmop{nilp}}^{(k)}).$$
Note that, from Remark~\ref{remark-invJordan}.(3)  we conclude that  
$$B(\partial^{(k)}) = B (\partial) / \mathfrak{m}^{(k)}.$$
In other words, the definition of the Bruno ideal {\em commutes} with the operation of truncation modulo $\mathfrak{m}^{(k)}$.  

The reduction to the formal normal form modulo its Bruno ideal will be carried out inductively with respect to $k$. Consequently, we assume that the following induction hypothesis holds for a given $k\in \mathbb{N}$:

\vspace{0.5cm}
\noindent {\bf Hypothesis $(\tmmathbf{H}_k)$ } $\partial\ku$ is in normal form modulo $B (\partial^{(k)})$.
\vspace{0.5cm}

\noindent We recall that this hypothesis corresponds to the following two requirements: \
\begin{description}
  \item[$(\tmmathbf{H}_k^{(1)})$] $B (\partial^{(k)})$ is $S$-invariant, and
  \item[$(\tmmathbf{H}_k^{(2)})$] We have the equality
  \[ (\partial^{(k)}_{\tmop{ss}})_{A^{(k)}} = S_{A^{(k)}}, \]
  in restriction to the quotient algebra $A^{(k)} = \mathcal{J}^{(k)} / B(\partial^{(k)}) .$
\end{description}
It follows from Corollary \ref{Corollary-coeffinB} that we can write
\begin{equation}
  (\partial^{(k)})_{A^{(k)}} = (1 + f_0^{(k)}) S_{A^{(k)}},
  \label{Aktruncation}
\end{equation}
for some $f_0  \in \tmop{Gr}_0 \left( \vO, S \right)$. Here $f_0^{(k)}
$ can be seen as the polynomial part of degree at most $2^k$ in the power series expansion of
$f_0$. \ Using this notation, let us consider the {\tmem{next}} quotient
algebra
\[ A^{(k + 1)} = \mathcal{J}^{(k + 1)} / B (\partial^{(k + 1)}), \]
associated to the Bruno ideal of $\partial^{(k + 1)}$, seen as a derivation in
the jet space $\mathcal{J}^{(k + 1)}$.  

It follows from (\ref{Aktruncation})
that the restriction of $\partial^{(k + 1)}$ to $A^{(k + 1)}$ can be written
as \ (\footnote{We have $B  (\partial^{(k + 1)}) / \mathfrak{m}^{(k)} = B
(\partial^{(k)})$, and therefore
\[     ({\mathcal{J}^{(k + 1)} / \mathfrak{m}^{(k)}})  \Big/ ( {B  (\partial^{(k + 1)}) /
   \mathfrak{m}^{(k)}}) = \mathcal{J}^{(k)} / B (\partial^{(k)}). \]
\ })
\begin{equation}\label{definition-of-W}
 (\partial^{(k + 1)})_{A^{(k + 1)}} = ((1 + f_0^{(k)}) S + W)_{A^{(k + 1)}}, 
\end{equation}
where $W$ is a polynomial derivation such that 
$$\ord(W) \ge 2^k, \quad \text{and}\quad \deg(W) \le 2^{k+1},$$
according to the definition in subsection \ref{subsect-ideals-expansions}. As previously, we write 
$$W = W_0 + W_\ast,$$
where $W_0$ is the degree $0$ component of $W$ with respect to the $S$-graduation. 
Let $U_k$ be the derivation satisfying the equations 
\begin{equation}\label{eq:homological}
    \big[ (1 + f_0) S, U_k \big] = - W_{\ast},\qquad U_k = (U_k)_\ast,
\end{equation}
whose existence and unicity has been established in Corollary \ref{cor-uniquesol-cohomological-eq}. The following result constitutes the key step in the inductive procedure.   
\begin{proposition}\label{prop:formalinductivestep}
Consider the automorphism $\Phi = \exp (U_k)$. Then, the conjugated
  derivation
  \[ \delta = \Phi\, \partial\, \Phi^{- 1}, \]
  is such that its truncation $\delta^{(k + 1)}$ satisfies hypothesis $(\tmmathbf{H}_{k +
  1})$. 
\end{proposition}
\begin{proof}
  It follows from Lemma~\ref{lemma-pull-back-B} that $B (\delta^{(k + 1)}) =
  \Phi^{\ast} B (\partial^{(k + 1)})$. Moreover, if we consider the quotient
  algebra ${\Lambda}^{(k + 1)} = \mathcal{J}^{(k + 1)} / B (\delta^{(k + 1)})$ with respect to Bruno ideal of $\delta$,  then we can
  write
  \begin{eqnarray*}
    (\delta^{(k + 1)})_{{\Lambda}^{(k + 1)}} & = & \Phi\, ((1 + f_0^{(k)}) S + W)_{A^{(k + 1)}}\, \Phi^{- 1}\\
    & = & (\Phi\, (1 + f_0^{(k)}) S\, \Phi^{- 1} + \Phi\, W\, \Phi^{-
    1})_{{\Lambda}^{(k + 1)}} \; .
  \end{eqnarray*}
  Since $\Phi = \exp (U_k)$ and $\ord(U_k) \ge 2^k$ (as it follows from Corollary~\ref{cor-uniquesol-cohomological-eq}.(2)), the following
  two equalities hold {\tmem{modulo $\mathfrak{m}^{(k + 1)}$}}, \
  \[ \Phi\, (1 + f_0) S\, \Phi^{- 1} = (1 + f_0^{(k)}) S + \big[ U_k, (1 +
     f_0^{(k)}) S \big] \qquad \tmop{mod} \quad \mathfrak{m}^{(k + 1)}, \]
  and
  \[ \Phi\, W\, \Phi^{- 1} = W \qquad \tmop{mod} \quad
     \mathfrak{m}^{(k + 1)}.\]
  Therefore, by the choice of $U_k$,
  \[ (\delta^{(k + 1)})_{{\Lambda}^{(k + 1)}} = ((1 + f_0^{(k)}) S + W_0)_{{\Lambda}^{(k + 1)}}. \]
  We now observe that $S$ {\em commutes} with $f_0^{(k)} S + W_0$. Moreover $S$ and $f_0^{(k)} S + W_0$ are respectively
  {\em semi-simple} and {\em nilpotent}. Hence, by Lemma \ref{Lemma-goestoquotient}, the
  ideal $B (\delta^{(k + 1)})$ is $S$-invariant and we have
  \[ S_{{\Lambda}^{(k + 1)}} = (\delta^{(k + 1)}_{\tmop{ss}})_{{\Lambda}^{(k + 1)}}. \]
  Therefore, the hypothesis \tmtextbf{$(H_{k + 1}^{(1)})$} and \tmtextbf{$(H_{k +
  1}^{(2)})$} hold for $\delta^{(k + 1)}$. 
\end{proof}
\begin{remark}\label{remark-truncated-bracket-equation}
\begin{enumerate}
    \item For later use, we note that, since the above proof only involves truncations
modulo \( \mathfrak{m}^{(k+1)} \), the same argument applies if \(U\) is replaced
by any solution of the following \emph{truncated bracket equation}:
\[
[(1+f_0)S,\, U_k] = -\,W_{\ast}
\qquad \mod \mathfrak{m}^{(k+1)}.
\]
In particular, \(U _k \) may be chosen to be a polynomial derivation of degree
strictly less than \(2^{k+1}\).  
\item  For the same reason, if \(U_k \) is assumed to be a polynomial solution of the
truncated equation above, we may replace the coordinate change defined by
\(\exp(U_k)\) with the {\em polynomial} coordinate change
\[
\varphi(x) = x + U_k(x),
\]
which corresponds to retaining only the first two terms of the exponential
series~\eqref{exponential-formula}. This is precisely the choice adopted in the
work of Bruno in \cite{Bruno1971_1972}.
\end{enumerate}
 
\end{remark}
Applying recursively the previous Proposition, we obtain the following {\em formal} normal form result:
\begin{thm}
   \label{thm:formalnormalizationtruncated}
		  Let $\partial$ be a formal derivation in $\tcD$. Then, there exist a formal automorphism $\Phi$ such that the conjugated derivation
          $$\delta= \Phi\, \partial\, \Phi^{-1}$$ 
          is in normal form modulo its Bruno ideal $B(\delta)$.
\end{thm}
\begin{proof}
    For each $k \in \N$, consider the automorphism $\Phi_k = \exp(U_k)$, where $U=U_k$ is defined by the Proposition \ref{prop:formalinductivestep}. Since $\ord(U_k) \ge 2^k$, it follows from the exponential series (\ref{exponential-formula}) that we can write
    $$
    \Phi_k = \mathbf{Id} + \phi_k,
    $$
    where $\phi_k$ is an endomorphism mapping $\tO$ to $\mathfrak{m}^{(k)}$. Hence, the sequence of automorphisms
    $$
    \Phi_k\circ \cdots \circ \Phi_1,\quad k\in \N
    $$
    converges (with respect to the Krull topology) to a formal automorphism $\Phi$. 
    
    Moreover, the resulting derivation $\delta = \Phi\, \partial\, \Phi^{-1}$ is such that, for all $k\in \N$, its truncation $\delta^{(k)}$ satisfies the hypotheses $\mathbf{(H_k)}$.  Therefore, by considering the (inductive) limit, we conclude that $\delta$ is in normal form modulo $B(\delta)$.
\end{proof}
\section{Analytic reduction to normal form modulo the Bruno ideal}\label{sec:analyticreduction}
In this section, we prove that under the $\omega$-condition, the formal reduction to normal form constructed in the previous section is analytic, provided that $\partial$ is analytic. 

As we mentioned in the introduction, the key analytic estimates rely on the original ideas of Bruno, detailed in \cite{Bruno1971_1972}. However, we will reformulate these estimates in terms of the natural $S$-graduation of the space of vector fields given by (\ref{graduation-S-der}), in the spirit of Martinet's survey \cite{Martinet1981Norm}.

For the sake of clarity, the next subsection is devoted to establishing some basic general facts about majorant norms for analytic germs and analytic derivations.
\subsection{Analytic $r$-norms}\label{subsect-analytic-norms}
We denote by $D_r = \{\x\in \mathbb \C^n: |x_i|<r, \ i=1,\ldots, n\}$ polydisk of radius $r>0$, and by $\OO(D_r)$ the ring of series which are absolutely convergent on the closure $D_r$. In other words, a series  
	$$ f= \sum_{\um\in \N^n} a_{\um} \x^{\um}  $$
	belongs to $\OO(D_r)$ if its {\em $r$-majorant norm} (or simply {\em $r$-norm}) 
    $$\| f\|_r:=\sum_{\um \in \N^n} |a_{\um }| r^{\|\um\|}$$ 
    is finite, where we note $\|\um \| = \sum_{i=1}^n m_i$. Observe that each $f \in \OO(D_r)$ is an analytic function on $D_r$, and hence defines an analytic germ. Reciprocally, every analytic germ admits a representative in $\oO(D_r)$ for some $r>0$. We also have
    $$
    \| f \cdot g \|_r \le \| f \|_r \cdot \| g\|_r,
    $$ 
which implies that $\OO(D_r)$ is a Banach algebra with respect to the above norm. 

    Similarly, we define the $r$-norm for a logarithmic derivation $\partial$ as follows. We consider its monomial expansion (see \eqref{monomial-exp-der})
   $$
   \partial = \sum_{\um\in \N^n} \x^{\um} L(\lambda_{\um}),$$
   and set 
   $$
   \|\partial\|_r := \sum_{\um \in \N^n}r^{\|\um \|} \, \|\lambda_{\um}\|,
   $$
   where, as previously, we define $\| \lambda \| = \sum_{i=1}^n |\lambda_i|$. We denote by $\cD(D_r)$ the space of logarithmic derivations with finite $r$-norm. As above, we observe that each $\partial \in \cD(D_r)$ defines a germ of analytic derivation, i.e.~an element of $\cD$.  We also have
 $$
    \| f \cdot \partial \|_r \le \| f \|_r \cdot \| \partial \|_r,
 $$
 for all $f \in \oO(D_r)$ and $\partial \in \cD(D_r)$.  
\begin{remark}\label{remark-norm-logarithmic-components}
Consider the expansion of $\partial$ with respect to an arbitrary logarithmic basis $\{L(\mu_0),\ldots,L(\mu_{n-1})\}$  
$$\partial = \sum_{i=0}^{n-1} g_i\, L(\mu_i),$$
as in subsection \ref{subsect-ideals-expansions}. Then it follows from Remark~\ref{passage-between-expasions} that for every $r>0$, there exist two constants $0<c<d$ (depending only on $\mu_0,\ldots,\mu_{n-1}$) such that
   $$c\;  \|\partial \|_r\leq \sum_{i=0}^{n-1}\|g_j\|_r \, |\mu_j|  \leq d\; \|  \partial \|_r\ . $$
In particular,  $\partial \in \cD(D_r)$ if and only if $g_0,\ldots,g_{n-1} \in \oO(D_r)$.
 \end{remark}
As for the series, notice finally that for any $\partial \in \mathcal D$, there is some representative in $\mathcal D (D_r)$ for a small enough $r>0$.

    We now list some useful properties of the $r$-norm.
	\begin{enumerate}
		\item Let  $\varphi=(\varphi_1, \ldots, \varphi_n ) \in \oO(D_\rho)^n$ be such that $\max \{\|\varphi_1\|_\rho, \ldots, \|\varphi_n\|_\rho\}\le r$ for some $\rho,r>0$. Then, $\varphi(D_\rho) \subseteq D_r$ and the composed function $f \circ \varphi=f(\varphi_1,\ldots, \varphi_n)$ satisfies
	$$
	\| f \circ \varphi \|_\rho \le \| f \|_r,
	$$
	for all functions $f \in \oO(D_r)$. 
		\item Suppose that $\ord(f) \ge k$ (i.e.~that $f\in \mathfrak{m}^{k}$), and let $0<\rho<r$. Then, we have the following relation between the $\rho-$norm and the $r-$norm.
		$$  \| f\| _\rho = \sum_{\um \in \N^n} |a_{\um}| \rho^{\| \um\| } = \sum_{\um \in \N^n} |a_{\um}| \left(\frac{\rho}{r}\right)^{\| \um\| }  r^{ \|m\| } \leq \left(\frac{\rho}{r}\right)^{  k } \| f\|_ r, $$
        where the last inequality follows from the fact that $a_{\um}$ vanishes for $\|\um\| < k$.
		\item Similarly for $\partial \in \cD(D_r)$ a derivation such that $\ord(\partial)\ge k$ and $0<\rho<r$, we have
		$$ \|\partial\|_\rho\leq \left(\frac{\rho}{r}\right)^{  k } \| \partial \|_ r . $$
		\item  Given a monomial derivation $x^{\um} L(\lambda)$ and a monomial $x^{\un}$, the identity \eqref{Lie-deriv-monomial} gives the estimate
		\begin{equation}
		\| x^\um L(\lambda)\, x^\un \|_r \leq |\langle \lambda,\un\rangle |\,  \|x^{\um+\un}\|_r \leq \|\un\|\,  \|x^\um L(\lambda)\|_r\, \|x^\un\|_r,
		\end{equation}
		 where we have used the trivial inequality $|\langle u,v \rangle | \leq \|u\|\, \|v\|$.\\
		 Similarly, given two monomial derivations $x^{\um} L(\lambda),\, x^{\un} L(\mu)$, the identity \eqref{Lie-bracket-formula} yields the following {\em key estimate} for their Lie bracket:
	\begin{equation*}
\begin{split}
\bigl\|\, \big[\, x^{\um} L(\lambda),\, x^{\un} L(\mu) \,\big]\, \bigr\|_r
&\le
\|\lambda\|\, \|\un\|\, \bigl\| x^{\um+\un} L(\mu) \bigr\|_r \\
&\quad
+ \|\mu\|\, \|\um\|\, \bigl\| x^{\um+\un} L(\lambda) \bigr\|_r \\
&\le
\bigl(\|\un\| + \|\um\|\bigr)\,
\bigl\| x^{\um} L(\lambda) \bigr\|_r\,
\bigl\| x^{\un} L(\mu) \bigr\|_r \; .
\end{split}
\end{equation*}
       \item More generally, we consider two derivations $H,K \in \cD(D_r)$ with monomial expansions
        $$
        H = \sum_{\um\in \N^n} \x^{\um} L(\lambda_{\um}), \quad
        K = \sum_{\un\in \N^n} \x^{\un} L(\mu_{\un}).
        $$
        Then, the following estimate holds for the Lie bracket $[H,K]$, 
        $$
        \|\, [H,K]\, \|_r \le \sum_{\um,\un} (\|\underline{m}\| + \| \underline{n} \|)\; \|\x ^{\um} L(\lambda_{\um})\|_r \,\|\x ^{\un} L(\mu_{\un})\|_r.
        $$
        In particular, if we assume that $H$ and $K$ are both polynomial derivations (see \eqref{def-degree-order-deriv}), then
		\begin{equation*}
		\begin{split}
			 \| [H, K]  \|_r  \le (\, \deg(H) + \deg(L)\, )\;
			\|  H \|_r\, \|  K \|_r.
		\end{split}
	\end{equation*}   
	\item Similarly, if $f = \sum_{\um} a_\um x^\um$ is a polynomial function and $\partial \in \cD(D_r)$ is an arbitrary derivation then
	\[
	\| \partial\, \big(f)\|_r \le \deg(f)\,  \| \partial \|_r\, \|f\|_r.
	\]
	\end{enumerate}
We now derive a simple estimate for the flow maps $(t,x) \mapsto \Phi_\partial(t,x)$ associated with analytic derivations $\partial$.  This estimate will be used to control the majorant norm under coordinate changes defined by such flows. 

Assuming 
that $\partial \in \cD(D_r)$, we say that the flow $\Phi_\partial$ is {\em defined up to time 1} at a point $x\in \C^n$ if 
$$
\Phi_\partial(t,x) \in D_r,
$$
for all complex times $t \in \mathbb{D} = \{z \in \C : |z| \le 1\}$. 
\begin{lemma}
Suppose that there exists a constant $0 \leq c < r$ such that
\[ \| \partial \|_r \leq c / r.\]
Then, the flow of $\partial$ is defined up to time 1 for all points lying in the polydisk $D_{r-c}$ of radius $r-c$.
\end{lemma}
\begin{proof}
Let us note by $\Phi_1(t,x),..,\Phi_n(t,x)$ the $n$-components of the flow $\Phi = \Phi_\partial$ of $\partial$. We observe that, for all $x \in D_r$ and all sufficiently small $t \in \C$ (depending on $x$), each  component $\Phi_i$ of $\Phi$ satisfies the integral equation
\begin{equation}\label{integral-equation}
\Phi_i(t,x) = x_i + \int_{0}^{t} f_i(\Phi_i(s,x))) ds,
\end{equation}
where the integral is computed over segment $[0,t]$ and $f_i = \partial(x_i)$ is the function obtained by applying $\partial$, as a derivation, to the coordinate $x_i$. 
We now observe, using the monomial expansion of $\partial$, that
\[
 f_i =  \sum_{\um\in \N^n} \x^{\um} L(\lambda_{\um}) x_i = \left(\sum_{\um\in \N^n} \x^{\um}\cdot\lambda_{\um,i} \right) x_i, 
\]
where $\lambda_{\um,i}$ denotes the $i^{\mathrm{th}}$ component of $\lambda_\um$. As a consequence,
\[
\|f_i\|_r \le r\, \|\partial\|_r \le c.
\]
Suppose now that the initial point $x$ belongs to the poly-disk $D_{r-c}$. Then, assuming that $\Phi(t,x)$ is defined for $|t|\le T$, the equation \eqref{integral-equation} gives
$$
|\Phi_i(t,x)| \le |x_i| + T c.
$$
Now, suppose the contrary: there exists a time $t_0$ with $|t_0|\le 1$ such that the solution curve $t \mapsto \Phi(t,x)$ exits the polydisk $D_r$ at some first exit time $t_1 \in [0,t_0]$. This immediately contradicts the inequality above.
\end{proof}
The following Corollary is an immediate consequence of the Lemma and property (1). It will allow to estimate the norm of a function under the coordinate change given by a time-1 flow.
\begin{corollary}\label{corollary-inclusion-of-disks}
Let $\partial$ be as above and let $x \mapsto \varphi(x) = \Phi_\partial(1,x)$ denote the time-1 flow of $\partial$. Then
\[
\varphi(D_{r-c}) \subset D_r.
\]
\end{corollary}
\noindent As a consequence, for any function $f\in \oO(D_r)$, the property (1) gives the estimate
\begin{equation}\label{estimate-f-flow}
\| \exp(\partial)(f) \|_{r-c} \le \| f \|_r, 
\end{equation}
since the action of the exponential automorphism $f \mapsto \Phi(f) = \exp(\partial)(f)$ (see \eqref{exponential-formula}) is precisely given by $f \mapsto f \circ \varphi$.

We conclude this subsection by studying how the $r$-norm of a derivation behaves under the coordinate change defined by $\Phi=\exp(\partial)$. We recall from subsection~\ref{subsect-ideals-expansions} that, for a derivation $\delta$, we can write
$$
\Phi \, \delta \, \Phi^{-1} = \exp(\ad_\partial) \delta = \sum_{n=0}^\infty \frac{1}{n!} \ad_\partial^n\, (\delta),
$$
where $\ad_\partial(\delta) = [\partial,\delta]$.  

We are particularly interested in the case where $\partial$ is a polynomial derivation and $\delta = L(\mu)$ is a diagonal derivation.  
Therefore, we start by considering the action of the iterated adjoint map
$$
\ad_\partial^n = \ad_\partial \circ \cdots \circ \ad_\partial,\qquad n\in \N
$$
in such case. The next enunciates refer to the order and degree of $\partial$, as defined in \eqref{def-degree-order-deriv}.
\begin{lemma}\label{lem:rnormadjoint}
    Let $\partial$ be a polynomial derivation. Then,
    $$\deg\big( \ad_\partial^n(L(\mu))\big) \le n\cdot \deg(\partial)\qquad\text{and}\qquad \frac{1}{n!}\| \,\ad_\partial^n(L(\mu))\, \|_r\leq \, \Big[\deg(\partial) \, (\|\partial\|_r)\Big]^n \|\mu\|,$$
for all $\mu \in \C^n$.
\end{lemma}
\begin{proof}
Let $d = \deg(\partial)$. Using the properties (4) and (5) listed above, we obtain 
$$
\deg\big( \ad_\partial(L(\mu))\big) \leq d, \qquad\text{and}\qquad \|\ad_\partial(L(\mu))\|_r \leq \, d\, \|\partial\|_r \, \|\mu\|.
$$
We now proceed by induction. Assuming that 
$$
\deg\big( \ad_\partial^{k-1}(L(\mu))\big) \le (k-1) d, \qquad\text{and}\qquad \|\ad_\partial^{k-1}(L(\mu))\|_r \leq \,(k-1)!\, d^{k-1}\, (\|\partial\|_r)^{k-1}\, \|\mu\|,
$$
the induction step follows immediately by applying again properties (4) and (5). This concludes the proof.
\end{proof}
The following Corollary is an immediate consequence of the Lemma and property (3):
\begin{corollary}\label{Corollary-estimate-deg-ord}
Let $\partial$ be a polynomial derivation. Then, for all $0<\rho<r$ and $n \in \N$, 
$$
\frac{1}{n!}\, \|  \ad_\partial^n (L (\mu))  \|_{\rho} \leq
\left[ \deg (\partial)  \left( \frac{\rho}{r}\right)^{\tmop{ord} (\partial)} \|\partial\|_r
\right]^n 
 \| \mu \|.
$$
\end{corollary}
\subsection{Bounds on $r$-norms under conjugation}	
The following technical lemma provides the first key ingredient in the induction procedure for analytic normalization. Its purpose is to control the norm of a derivation under a coordinate change defined by the flow of a polynomial vector field. In general, such a norm cannot be controlled on a polydisk of fixed radius, since it tends to {\em blow up} as one approaches the boundary. The central idea is therefore to work with three carefully chosen nested polydisks, under the constraint that their radii do not shrink too much.

We emphasize that this result is a direct consequence of several technical lemmas established in \cite{Bruno1971_1972}. For the sake of completeness, we state it separately and provide an independent proof.

Let $C\geq 1$ be a constant and $(\Omega_k)_k$ be some sequence of real numbers satisfying $\lim_{k \rightarrow \infty} \Omega_k = 1$, both the constant and the sequence to be specified later. For every $k\in \mathbb N$, let $0
< \rho_1 < r < \rho$ be three radii such that 
\begin{equation} \label{radius-assumptions}
	\frac{\rho_1}{r} = \left( \frac{1}{C 2^k} \right)^{C / 2^k}, \quad
	\frac{r}{\rho} = \left( \frac{1}{C 2^k} \right)^{C / 2^k} \Omega_k ,
\end{equation}
and let $U\in \cD(D_r)$ be a polynomial derivation satisfying the following conditions
\[ \ord (U) \geq 2^k, \quad \deg (U) \leq 2^{k + 1}, \quad \| U \|_r
\leq \frac{1}{2^k}. \]

\begin{lemma}\label{lemma-bound-expU}
  Under the above notation, there exists a $k_0 \in \mathbb{N}$ $($depending only on $C$ and the sequence $(\Omega_k)\, )$ such that,
  for all integers $k \geq k_0$, the following conditions hold:
  \begin{enumeratenumeric}
    \item $\varphi (D_{\rho_1}) \subset D_r$, where $\varphi$ is the time-1
    flow of $U$. 
    \item For an arbitrary function $f \in \oO (D_{\rho})$, \ \
    \[ \| \exp (U) f \|_{\rho_1} \leq \| f \|_{\rho}. \]
    \item If $R \in \cD (D_{\rho})$ is a derivation with $\ord (R) \geq
    1$ then
        \[ \| \exp (\tmop{ad}_U) R \|_{\rho_1} \leqslant \left( 1 - \frac{C k}{2^k}
   \right)  \| R \|_{\rho}. \]
 \end{enumeratenumeric}
\end{lemma}
\begin{proof}
  The item $(1)$ will follow directly from Corollary
  \ref{corollary-inclusion-of-disks} if we show that $\rho_1 \leq r - r
  \| U \|_r$. Under the above hypothesis on $\|U\|_r$, this is equivalent to prove the inequality
  \[ 1 - \left( \frac{1}{C 2^k} \right)^{C / 2^k} \geq \frac{1}{2^k}. \]
  But notice that the left-hand side of this inequality is equivalent to
  $\frac{k C\log 2}{2^k}$ as $k \rightarrow \infty$. Therefore, the
  inequality holds for all sufficiently large $k$.
  
  The item $(2)$ is an immediate consequence of item (1) and equation
  (\ref{estimate-f-flow}).
  
  Let us prove item (3). We initially consider the case of a monomial
  derivation $R = x^{\um} L (\mu)$ and observe that we can write
  \[ \exp (\ad_U) \left( x^{\um} L (\mu) \right) = \exp (U) x^{\um}
     \cdot \exp (\ad_U) L (\mu). \]
  Therefore, applying separately item (2) of the present Lemma to the monomial $x^{\um}$ and Corollary
  \ref{Corollary-estimate-deg-ord} to the derivation $L (\mu)$, we obtain
  \begin{eqnarray*}
    \left\| \exp (\ad_U) \left( x^{\um} L (\mu) \right)
    \right\|_{\rho_1} & \leq & \left\| \exp (U) x^{\um} \right\|_{\rho_1}
    \cdot \| \exp (\ad_U) L (\mu) \|_{\rho_1}\\
    & \leq & \left\| x^{\um} \right\|_r  \sum_{n \geq 0} \left[
    \deg (U)  \left( \frac{\rho_1}{r} \right)^{\tmop{ord} (U)}  \| U \|_r
    \right]^n \| \mu \|\\
    & \leq & \left\| x^{\um} \right\|_r \sum_{n \geq 0} \left[ 2
    \left( \frac{1}{C 2^k} \right)^C  \right]^n\, \| \mu \|.
  \end{eqnarray*}
  We can assume that the constant $A = 2 \left( \frac{1}{C 2^k} \right)^C$ is
  $< 1$ by choosing $k$ sufficiently large. Therefore, applying property (1)
  to $x^{\um}$ and summing the geometric series, we obtain the estimate
  \[ \left\| \exp (\ad_U) \left( x^{\um} L (\mu) \right)
     \right\|_{\rho_1} \leq \frac{\left( r/\rho \right)^{\left\|
     \um \right\|}}{1 - A}  \left\| x^{\um} \right\|_{\rho} \| \mu \| =
     \frac{\left( r/\rho \right)^{\left\| \um \right\|}}{1 - A} 
     \left\| x^{\um} L (\mu) \right\|_{\rho} . \]
  In the general case, under the hypothesis that $\ord(R) \geq 1$, we
  can write the monomial expansion of $R$ as
  \[ R = \sum_{\left\| \um \right\| \geq 1} x^{\um} L \left( \lambda_{\um}
     \right). \]
  Applying the above estimate to each monomial derivation $x^{\um} L \left(
  \lambda_{\um} \right)$, we conclude that
  \[ \| \exp (\ad_U) R \|_{\rho_1} \leq \frac{r / \rho}{1 - A}  \|
     R \|_{\rho}. \]
Using that $\frac{1}{1 - A} \leqslant 1 + 2 A$ for $A \leqslant 1 / 2$ we obtain the equivalence
\[ \frac{r / \rho}{1 - A} \leqslant \Omega_k \frac{1}{(C 2^k)^{C / 2^k}}
   \left( 1 + 4 \frac{1}{C^C 2^{k C}} \right) = 1 - \frac{C k \log 2}{2^k}(1+o(1)),
\]
where $o(1)$ indicates a term which goes to zero as $k \rightarrow \infty$. As a consequence, there exists a $k_0 \in \N$ such that
\[ \| \exp (\tmop{ad}_U) R \|_{\rho_1} \leqslant \left( 1 - \frac{C k}{2^k}
   \right)  \| R \|_{\rho}, \]
for all $k \ge k_0$.      
\end{proof}
We now suppose that $\partial \in \cD(D_\rho)$ is a $S$-perturbation (see Definition \ref{def-S-perturb}) for some linear diagonal $S\neq 0$. In other words, that $\partial$ can be written as
$$
\partial = S + R,
$$
where $R$ is a derivation with $\ord(R) \ge 1$. Note that the conjugated derivation $\partial_1 = \exp (\ad_U) \partial $ also has the form of a $S$-perturbation. Indeed, we can write $\partial_1 = S + R_1$  with
$$
R_1 = (\exp (\ad_U)-\mathbf{Id}) S + \exp (\ad_U) R,
$$
which implies that $\ord(R_1) \ge \min\{\ord(U),\ord(R)\} \ge 1$. Using the previous Lemma, we obtain the following: 
\begin{corollary}\label{corollary-to-lemma-expU}
Suppose given a constant $\Delta >0$. Then, under the above assumption, there exists a constant $k_0 \in \N$ $($ depending only on $\lambda,\Delta,C$ and the sequence $(\Omega_k)$ $)$ such that
$$
\| R \|_{\rho} \le \Delta \Rightarrow  \| R_1 \|_{\rho_1} \le \Delta,
$$
for all $k \ge k_0$. 
\end{corollary}
\begin{proof}
Using the estimates of Corollary \ref{Corollary-estimate-deg-ord} and the fact that $S = L(\lambda)$ we have
 \[
\left\| \bigl(\exp(\operatorname{ad}_U)-\mathbf{Id}\bigr) S  \right\|
\le
\frac{A}{1-A}\,\|\lambda\| ,
\]
where $A =2 \left( \frac{1}{C 2^k} \right)^C$ as in the previous Lemma. In particular, 
$$
\frac{A}{1-A} \|\lambda\| \sim 2 \left( \frac{1}{C 2^k} \right)^C \| \lambda \| ,
$$
as $k \rightarrow \infty$. As a consequence, there exists a value of $k_0$ (depending on $\Delta$) such that
$$
\frac{A}{1-A} \|\lambda\| \le \frac{C k}{2^k} \Delta,
$$
for all $k \ge k_0$. Therefore, based on the item (3) of the previous Lemma, one obtains
$$
\| R_1 \|_{\rho_1} \le \frac{A}{1-A}\,\|\lambda\| + \left( 1 - \frac{C k}{2^k} \right)  \| R \|_{\rho} \le \frac{C k}{2^k} \Delta + \left( 1 - \frac{C k}{2^k} \right) \Delta = \Delta
$$
for sufficiently large values of $k$. 
\end{proof}\subsection{Analytic estimates for the bracket equation}
We now provide the second main ingredient in the induction procedure for analytic normalization. The goal is to control the norm of the solution $U$ of the bracket equation (\ref{eq:homological}) in terms of the norm of $W$. The $\omega$-condition on the eigenvalues of $S$ will be one of the essential ingredients here (even if it is not used in its {full-force}).

Let $S=L(\lambda)$ be a diagonal derivation satisfying the $\omega$-condition. Given $k \in \mathbb{N}$, we consider again three radii $0 < \rho_1
< r < \rho$ (depending on $k$) defined by (\ref{radius-assumptions}), where we now additionally suppose that
$$C \ge 3,\qquad \Omega_k : = (\omega_k)^{C / 2^k},$$ 
where $(\omega_k)_k = (\omega_k(\lambda))_k$ is the sequence defined in the Introduction in terms of the eigenvalues of $S$. 
Note that the $\omega$-condition implies that
$$\lim_{k \rightarrow \infty} \Omega_k = 1.$$ 
\begin{lemma}\label{lemma-bound-bracket-equation}
  Let $k\geq 0$ and $f_0 \in \Gr_0(\OO,S)$. Suppose that $W$ is a polynomial derivation satisfying the following conditions:
  \[ \ord (W) \geqslant 2^k, \quad \deg (W) \le 2^{k + 1}. \]
  Then, there exists a unique polynomial derivation $U$ which solves the truncated
  bracket equation
  \begin{equation}\label{truncated-bracket-eq}
  [(1 + f_0) S, U] = - W_{\ast} \qquad \tmmathbf{\tmop{mod} \quad
     \mathfrak{m}^{(k + 1)}} ,
  \end{equation}
  and {{additionally}} satisfies the following three conditions
  \[ \ord (U) \geqslant 2^k,\qquad \deg (U) < 2^{k + 1}\quad \text{and} \quad U =
     U_{\ast} .\]
  Moreover, there exists a $k_0 \in \mathbb{N}$ $($ depending only on
  $\| \lambda \|$ and the sequence $(\omega_k)$ $)$ such that, for all $k \geqslant k_0$, the conditions
  \[ \| f_0 \|_{\rho} \leqslant \frac{1}{2} \quad\text{and}\quad \| W \|_{\rho}
     \leqslant 1   \]
imply that $\| U \|_r \leqslant \frac{1}{2^k}$. 
\end{lemma}
\noindent 
\begin{remark}
Notice that we only require $U$ to solve the bracket equation up to flat terms of order $2^{k+1}$.  This is the reason why we can obtain a polynomial solution. 
\end{remark}
\begin{proof}
  We have seen in Corollary \ref{cor-uniquesol-cohomological-eq} that the
  equation
  \[ 
  \big[ (1 + f_0) S, V \big] = -W_{\ast} \]
  has a unique solution $V \in \cD$ such that $\tmop{ord} (V) \geqslant 2^k$
  and $V = V_{\ast}$.\quad Let us write the monomial expansion of such
  solution $V$ as
  \[ V = \sum_{\left\| \um \right\| \geqslant 2^k} x^{\um} L \left(
     \lambda_{\um} \right). \]
  Then, the derivation defined by its {\tmem{truncation at degree $2^{k +
  1}$}},
  \[ U = \sum_{2^k \leqslant \left\| \um \right\| < 2^{k + 1}} x^{\um} L
     \left( \lambda_{\um} \right).\]
  is the unique polynomial solution satisfying both the bracket equation modulo
  $\mathfrak{m }^{(k + 1)}$, and the required conditions on the order and the
  degree stated in the enunciate.
  
  It remains to show the bound on the $r$-norm $\| U \|_r$. We initially consider the case where $W = x^{\um}  L (\mu)$ is a monomial derivation
  belonging to $\Gr_{\alpha} \left( S, \cD \right)$, i.e.~such that $\langle \lambda,\um\rangle = \alpha$. Then, from the
  explicit formula for $\tmop{ad}_{(1 + f_0) S}^{-1}$ given in Lemma
  \ref{lemma-formula-inverse-adjoint} we obtain the estimate
  \begin{eqnarray*}
    \left\| \ad_{(1 + f_0) S}^{- 1}  \left( \left. x^{\um}  L (\mu) \right)
    \right) \right\|_{\rho} & \leqslant &  \left(  \frac{\| \mu \|}{\alpha (1
    - \| f_0  \|_\rho)} + \frac{\| L (\mu) f_0^{(k)} \|_{\rho}}{\alpha^2  (1 - \|
    f_0 \|_\rho)^2}  \| \lambda \|  \right) \left\| x^{\um} \right\|_{\rho}\\
    & \leqslant &  \left(  \frac{1}{\omega_k  (1 - \| f_0  \|_\rho)} + \frac{2^{k
    } \| f_0 \|_{\rho}}{\omega_k^2 (1 - \| f_0 \|)^2}  \| \lambda \|
    \right)  \left\| x^{\um} \right\|_{\rho} \| \mu \| \\
    & \leqslant & \left(  \frac{2}{\omega_k } + \frac{2^{k + 1}}{\omega_k^2} 
    \| \lambda \| \right) \left\| x^{\um} L (\mu) \right\|_{\rho}\\
    & \leqslant & \left( \frac{2}{\omega_k} \right)  \left( 1 + \frac{2^{k }}{\omega_k}   \| \lambda \| \right) \left\| x^{\um} L (\mu)
    \right\|_{\rho},
  \end{eqnarray*}
  where we used property (6) of subsection \ref{subsect-analytic-norms} to
  write that
  \[ \| L (\mu) f_0^{(k)} \|_{\rho} \leqslant \| \mu \| \deg (f_0^{(k)})  \|
     f_0^{(k)} \| \leqslant \| \mu \| 2^{k } \| f_0 \|_{\rho}. \]
  We consider now the case where $W$ is an arbitrary polynomial derivation satisfying the conditions of the enunciate. Note that we
  can write $S$-weighted decomposition of $W_{\ast}$ as
  \[ W_{\ast} = \sum_{\alpha \neq 0} W_{\alpha}, \]
  where each $W_{\alpha} \in \tmop{Gr}_{\alpha} \left( S, \cD \right)$ is a
  finite sum of monomial derivations. Applying the above estimate, and the
  triangle inequality, we obtain
  \[ \| U \|_{\rho} \leqslant \left( \frac{2}{\omega_k} \right)  \left( 1 +
     \frac{2^{k + 2}}{\omega_k}   \| \lambda \| \right)  \| W \|_{\rho}. \]
  Since $\ord (U) \geqslant 2^k$, it follows from property (3) of subsection
  \ref{subsect-analytic-norms} that
  \[ \| U \|_r \leqslant \left( \frac{r}{\rho} \right)^{2^k} \| U \|_{\rho}. \]
  In view of the definition of the ratio $r/\rho$ in \eqref{radius-assumptions}, we get
  \[ \| U \|_r \leqslant \left( \frac{\omega_k}{C 2^k} \right)^C \left(
     \frac{2}{\omega_k} \right)  \left( 1 + \frac{2^{k }}{\omega_k}   \|
     \lambda \| \right) \sim \frac{2 \omega_k^{C - 2} }{ C^C} \| \lambda \|
     \frac{1}{2^{k (C - 1)}}. \]
 Since $C \ge 3$ and $\omega_k \le \|\lambda\|$ for all $k$, we conclude that $\| U  \|_r \leqslant
  \frac{1}{2^k}$ for all sufficiently large values of $k$. 
\end{proof}

\subsection{Proof the Main theorem: Analytic normalization modulo the Bruno ideal}

We will fix once and for all a logarithmic basis $\{ S, T_1, \ldots, T_n
\}$ so that each analytic derivation $R$ has an expansion
\[ R = f S + \sum_{j = 1}^{n - 1} g_j T_j,\]
with $f, g_j$ analytic germs. Further, it follows from Remark
\ref{remark-norm-logarithmic-components} that there exists an absolute
constant $d > 0$ (depending only on the choice of the basis $\{ S, T_1,
\ldots, T_{n - 1} \}$) such that for all derivations $R$ and all radius $\rho
> 0$ we have
\begin{equation}
  \max \{ \| f \|_{\rho}, \| g_1 \|_{\rho}, \ldots, \| g_{n - 1} \|_{\rho} \}
  \leqslant d \| R \|_{\rho}. \label{estimates-coeffs}
\end{equation}
In order to start the inductive procedure, we chose the following
constants
\[ \text{$\Delta = \min \left\{ 1, \frac{1}{2 d} \right\}$, $C = 3$}, \]
and choose $k_0 \in \mathbb{N}$ as the maximum of the integers $k_0$ given by
Corollary 4.7 and Lemma 4.8 (once we have fixed the above choices of $\Delta$ and
$C$).

Consider now an analytic $S$-perturbation $\partial = S + R$. We remark that the derivations obtained at each step of the induction procedure are $S$-perturbations. Before applying
the analytic normalization procedure, we need to perform the following
preparation steps:
\begin{enumerate}
  \item Using Remark \ref{remark-truncated-bracket-equation}, we can
  assume, up to a polynomial change of coordinates, that the hypothesis
  $\mathbf{(H_{k_0})}$ of subsection 3.2 holds for $\partial$, i.e. that the
  $\mathfrak{m}^{2^{k_0}}$-truncation of $\partial^{(k_0)}$ is in normal form
  modulo $B (\partial^{(k_0)})$.
  
  \item Up to a homothety, we can assume that $R \in \cD (D_{\rho})$ with
  $\rho = 1$.
  
  \item Up to replacing $R$ by some multiple $(1 / \mu) R$ with $\mu > 0$, we
  assume that
  \[ \| R \|_{\rho} \leqslant \Delta. \]
\end{enumerate}
Under these conditions, the following properties
hold for $k=k_0$:
\begin{description}
  \item[$\tmmathbf{P_k (1)}$] If we denote by $f_0$ the degree-0 component of $f$ with
  respect to the $S$-graduation then
  \[ \| f_0 \|_{\rho} \leqslant \| f \|_{\rho} \leqslant \frac{1}{2}. \]
  Note that the leftmost inequality follows simply from the fact that the power series
  of $f_0$ is extracted from the power series of $f$ and the rightmost
  inequality follows from (\ref{estimates-coeffs}). 
  \item[$\tmmathbf{P_k (2)}$] If we denote by $W_{\ast}$ the polynomial derivation
  defined by \eqref{definition-of-W} then
  \[ \| W_{\ast} \|_{\rho} \leqslant \| R \|_{\rho} \leqslant 1. \]
  As above, the leftmost inequality follows from the fact that the monomial
  expansion of $W_{\ast}$ is extracted from the monomial expansion of $R$.
\end{description}
As a consequence, we are precisely in the conditions of Lemma
\ref{lemma-bound-bracket-equation}, which guarantees that the polynomial
solution $U$ of the truncated bracket equation \eqref{truncated-bracket-eq}
satisfies the estimate
\[ \| U \|_r \leqslant \frac{1}{2^k}. \]
Therefore, we can apply Corollary \ref{corollary-to-lemma-expU} with
this choice of $U$. It guarantees that under the coordinate change given by
the automorphism $\Phi = \exp (U)$ one obtains a conjugated derivation $\partial_1 = S + R_1$ such that
\[ \| R_1 \|_{\rho_1} \leqslant \Delta, \]
where we recall that $\rho_1$ is given in terms of $\rho$ by
\eqref{radius-assumptions}.

Expanding the new derivation $\partial_1$ with respect to the fixed logarithmic basis $\{S,T_j\}$, one concludes that it satisfies the conditions $\tmmathbf{P_k (1)}$
and $\tmmathbf{P_k (2)}$ with now $k = k_0 + 1$. Therefore, the induction procedure can be
continued.

As in the proof of Theorem \ref{thm:formalnormalizationtruncated}, this inductive procedure produces a sequence of automorphisms 
$$\Phi_1,\;\Phi_1 \circ \Phi_2,\ldots,\; \Phi_1 \circ \cdots \circ \Phi_s,\ldots$$ 
which converges (in the Krull topology) to a formal automorphism $\Phi$. However, the above analytic estimates allows also to consider the associated sequence of coordinate changes
\[ \varphi_1 : D_{\rho_1} \rightarrow D_{\rho} \]
\[ \varphi_1 \circ \varphi_2 : D_{\rho_2} \rightarrow D_{\rho} \]
\[ \varphi_1 \circ \cdots \circ \varphi_s : D_{\rho_s}
   \rightarrow D_{\rho} , \quad s\ge 1,\]
where $\rho_s$ is given by
\[ \rho_s = \rho \prod_{k = k_0}^{k_0 + s} \left( \frac{\omega_k}{C 2^{2 k}}
   \right)^{\frac{C}{2^k}}.  \]
To conclude the proof, it suffices to show that the sequence of radii $(\rho_s)_{s\ge 1}$
converges to a strictly positive limit $\overline{\rho}$. Indeed, this will imply that the formal automorphism $\Phi$ constructed above defines an analytic coordinate change with domain the polydisk $D_{\overline{\rho}}$. 

Since the general term of the above product converges to $1$, it is enough to establish
the convergence of the series
\[ \sum_{k \geqslant k_0}  \frac{\log (\omega_k)}{C 2^k} - \sum_{k \geqslant
   k_0} \frac{2 Ck \log 2 + C \log C}{2^k}. \]
This follows immediately from the $\omega$-condition. The proof of the Main Theorem is therefore complete.

\section{Examples and applications of the Bruno ideal}\label{sec:examplesapplications}
We now present applications of the Main Theorem and its corollary, starting with two-dimensional vector fields.
\begin{example}
Let \(\partial = S + R\) be a formal logarithmic vector field with linear part
\[
S = \lambda\, x \xx + \mu\, y \yy,
\]
where \(\lambda \in \C^\ast\) and \(\mu \in \C\).
Assuming that \(\partial\) is in normal form, the perturbation term \(R\) can be written as
\[
R = GS + F \, y \yy,
\]
where \(F,G \in \C[[x,y]]\) are formal power series without constant term having an expansion of the form
\[
F = \sum_{\substack{(k,l)\in \N^2 \\ \lambda k + \mu l = 0}} a_{k,l} x^k y^l, \quad G = \sum_{\substack{(k,l)\in \N^2 \\ \lambda k + \mu l = 0}} b_{k,l} x^k y^l  .
\]
The Bruno ideal is therefore given by
$B(\partial) = \big\langle F \big\rangle$  and we distinguish two cases:

\begin{itemize}
\item \emph{Resonant case.}
If \(\mu/\lambda \in \Q_{\ge 0}\), then writing \(\mu/\lambda = -m/n\) in irreducible terms, we have
\[
F, G \in \C[[x^n y^m]] \cap \mathfrak{m}.
\]

\item \emph{Non-resonant case.}
If \(\mu/\lambda \notin \Q_{\ge 0}\), then \(F = 0\), \(G = 0\).
\end{itemize}

Let us now consider an arbitrary analytic vector field \(\delta\) that is formally conjugated to \(\partial\), and examine the information provided by the Main Theorem.

Assume first that \(\partial\) is in the resonant case. In this situation, the \(\omega\)-condition is always satisfied, and therefore
\[
V(B(\delta)) \subset (\C^2,0)
\]
is a germ of an analytic variety.

We further distinguish two sub-cases.  
If \(F = 0\), then \(V(B(\delta)) = (\C^2,0)\), and, up to an analytic change of coordinates, we may write
\[
\delta = u \left( n x \xx - m y \yy \right),
\]
for some analytic unit \(u \in \C\{x^n y^m\}\). On the other hand, if \(F \neq 0\), then \(V(B(\delta))\) defines a germ of an analytic variety with either one irreducible smooth component (if \(\mu = 0\)) or two irreducible smooth components (if \(\mu \neq 0\)), intersecting transversely.

Finally, assume that \(\partial\) is in the non-resonant case and that the \(\omega\)-condition holds. Then \(V(B(\delta)) = (\C^2,0)\), and the Main Theorem implies that \(\delta\) is analytically linearizable.
\end{example}
We observe that the resonant case above with $F=0$ corresponds to a classical theorem of Dulac: if a germ of a two-dimensional vector field with non-vanishing linear part admits a formal first integral, then it also admits an analytic first integral.

More interesting examples arise in higher dimension.
The following result provides an alternative proof of Lyapunov's theorem for logarithmic vector fields. 
Consider a germ of real analytic vector field of the form 
$$\partial = -y\xx + x\yy + \lambda z\zz + R,$$
where $\lambda \in \R^\ast$ and $R$ is $1$-flat. By the Center Manifold Theorem, $\partial$ always has a formal invariant manifold $\widetilde{W}^c$ tangent to the $z=0$ direction, called its {\em center manifold}. We say that $\partial$ if formally integrable on its center manifold if the restriction of $\partial$ to $\widetilde{W}^c$ has a formal fist integral. On the other hand, we say that $\partial$ is analytic integrable on its center manifold if $\widetilde{W}^c$ is a germ of analytic variety and the restriction of $\partial$ to $\widetilde{W}^c$ has an analytic first integral. 
\begin{proposition}[Lyapunov's theorem \cite{Lyapunov}] \label{prop:lyapunov}
    $\partial$
is formally integrable on its center manifold if and only if it is analytically integrable on its center manifold.
\end{proposition}
\noindent 
\begin{proof}
   Suppose that $B(\partial)=0$, then the result is obvious since the foliation is necessarily linearizable, being $z=0$ its analytically integrable center manifold. Now suppose $B(\partial)\neq 0$ and that the vector field is formally integrable in $\widetilde
   {W}^c$. Then it follows from the definition of $B(\partial)$ that $\widetilde{W}^c$ is an irreducible component of $B(\partial)$, therefore it is analytic.

\end{proof}
Related to Lyapunov's theorem, we get also a relatively easier proof for the following result, see~\cite{Aulbach1985,Corral2025,Kelley1967}.
\begin{prop}
    Let $\partial$ be a vector field in the hypotheses of the previous Proposition. There is an infinite family $\{C_i\}_{i\in I}$ of cycles accumulating at the origin if and only if $\partial$ is analytically integrable on its center manifold.
\end{prop}
\begin{proof}
    We only prove the direct implication since the converse is a direct consequence of the previous result. We note that there are many steps in common with respect to the proof in~\cite{Corral2025}. Suppose that there is an infinite family $\{C_i\}_{i\in I}$ of cycles accumulating at the origin. 
    Since the eigenvalues of the linear part of $\partial$ are $\lambda,\pm i$, there is an analytic one-dimensional stable (unstable) manifold $\Omega$ associated to $\lambda<0$ ($\lambda>0$) and a formal two-dimensional center manifold associated to $i,\ -i$.
    
    We start making a cylindric blow-up centered at the stable manifold. The fiber $\pi^{-1}(0)$ is the only cycle defined by the strict transform $\widetilde \partial$ of $\partial$ in $\pi^{-1}\Omega$. Based on this cycle, we can define the Poincaré map on an analytic transverse section $\Delta$. The Poincaré map $P:\Delta' \subset \Delta \to \Delta$ is an analytic map given by the first intersection of the trajectory starting at a point $p\in \Delta'$ with $\Delta$. 
    We will study the cycles of $\widetilde{\partial}$ near $\pi^{-1}(0)$ by means of studying the periodic points of the Poincaré map, since they are one-to-one related.
    
    The map $P$ has eigenvalues $1, e^{\lambda}$. Therefore, it has a one-dimensional  stable (unstable) manifold associated to the eigenvalue $e^{\lambda}<1$ ($e^{\lambda}>1$), and for each $k\in \mathbb N$ a $C^k$ one-dimensional center manifold $\Gamma_k$ associated to the eigenvalue $1$. 
    By the properties of the center manifolds~\cite{Hirsch1970}, the curve $\Gamma$ contains the periodic points of $P$. On the other hand, the saturation of such curve provides an invariant surface, which means that the cycles are contained in a surface, and consequently, they make a single turn and they must be one-to-one related to the fixed points of $P$. Then, since $\{C_i\}_{i\in I}$ is an infinite family of cycles accumulating to $0$, the Poincaré map has an infinite family of fixed points accumulating to $\pi^{-1}(0)\cap \Delta$. Since $\textrm{Fix}(P)\subseteq \Gamma_k$ is an analytic set, we have that $\textrm{Fix}(P)$ must be a curve, and its saturation provides an analytic surface $X$. 
    
    The surface $X$ is also invariant for its semi-simple part $\partial_{\textrm{ss}}$ of $\partial$. Notice that $\partial_{\textrm{ss}}$ only has two possible invariant manifolds: one curve tangent to the eigenspace of $\lambda$ and a surface tangent to the eigenspaces of $i,-i$. By the two-dimensional Dulac's theorem, we conclude  the analytic integrability of $\partial$ in $X$.
\end{proof}

Suppose that the linear part of $\partial$ is as in Proposition~\ref{prop:lyapunov} but now supposing that $\lambda=0$. Since there is a zero eigenvalue, more intricate Bruno varieties can arise.

\begin{example}
    Let $$\partial = ix\xx - iy\yy + (xy-z^2)(x\xx + y\yy) + (xy-z^2) z\zz $$
    be a vector field written on the logarithmic basis $S= ix\xx - iy\yy, \ T_1=x\xx + y\yy, \ T_2= z\zz$. As in the previous examples, $\partial$ is already in normal form, being $\partial_{\tmop{ss}}= ix\xx - iy\yy$ and $\partial_{\tmop{nilp}}= (xy-z^2)(x\xx + y\yy) + (xy-z^2) z\zz. $
    We have $\partial_{\tmop{ss}}\wedge \partial_{\tmop{nilp}}= (xy - z^2) (S\wedge T_1 + S\wedge T_2), $ thus
    $$  B(\partial) = \big \langle xy - z^2\big \rangle.  $$
    It produces an invariant analytic variety given by $V(B(\partial))=\{ xy - z^2= 0\}, $ which is a surface in $\C^3$. 

    This example has been carefully chosen in order to show the existence of surfaces filled with cycles. In particular, notice that the real vector field $$ \delta= -y\xx + x\yy  +  (x^2+ y^2-z^2)(x\xx + y\yy) + (x^2+ y^2-z^2) z\zz ,$$ 
    is conjugated to $\partial$ by the automorphism $\Phi(x,y,z)=(x+iy, x-iy,z)$. We observe that the real analytic surface $V(B(\delta))= V(\Phi^* (B(\partial)))= \{ x^2+ y^2-z^2=0\}  $ is filled with cycles.
\end{example}

We conclude this section by considering the quasi-periodic behavior of the foliation defined by $\partial$ in restriction to its Bruno ideal.

Let us assume that $\partial = S+R$ is an analytic $S$-perturbation such that $B(\partial)$ is analytic. Considering the quotient algebra $A = \OO / B(\partial)$, it follows from Corollary \ref{cor:coeffsinB} that we can write
$$\partial_A = (1+f_0) S_A,$$ 
for some analytic germ $f_0 \in \Gr_0(\OO,S) \cap \mathfrak{m}$. 

Therefore, following \cite{Bruno1989}, one can further consider the {\em period fibration} on $V(B(\partial))$ defined by $\{f_0=\mathrm{cte}\}$. It is easy to prove that each level set 
\begin{equation}\label{level-sets}
F_a = V(B(\partial)) \cap f_0^{-1}(\{a\})
\end{equation}
is a $\partial$-invariant analytic set. Moreover, in restriction to $F_a$, the local flow of $\partial$ is simply given by
$$
\Phi(t,x_1,\ldots,x_n) = (e^{t(1+a)\lambda_1} x_1,\ldots, e^{t(1+a)\lambda_n} x_n).
$$
Let $\partial=S+R$ be a real analytic derivation such that
    $$S= \omega_1\left(-y_1 \frac{\partial}{\partial x_1} + x_1 \frac{\partial}{\partial y_1}\right)  + \cdots + \omega_n \left( -y_n \frac{\partial}{\partial x_n} + x_n \frac{\partial}{\partial y_n} \right), $$ 
    with $\omega_1,\ldots,\omega_n\in \R^\ast$ and $R$ 1-flat. Suppose further that the vector 
    $$\lambda = (\pm\omega_1,..,\pm \omega_n) \in \R^{2n}$$ satisfies the $\omega$-condition. 
    Then it follows from the Main theorem that $B(\partial)$ is analytic and that $V = V(B(\partial))$ is a (germ of) analytic variety.
    \begin{lemma}
   For each point $
    p\in V$ there exists a constant $a\in \R$ such that the orbit of $\partial$ through $p$ lies in a quasi-periodic torus $\mathbb{T}_p \subset V$ with a set of frequencies in $\{ ((1+a)\omega_1,\ldots,(1+a) \omega_n)\}$.
\end{lemma}
\begin{proof}
    Notice that the foliation defined by $S$ fills the space with invariant tori, i.e. each trajectory is a quasi-periodic orbit contained in a tori of dimension ranging from 0 to $n$. In fact, $k-$dimensional tori for $k<n$ are contained in some $2k-$plane $\{ x_{i_1}=y_{i_1}=\ldots= x_{i_{n-k}}=y_{i_{n-k}} =0  \}$ for $1\leq i_1< \cdots < i_{n-k}\leq n$.
 Each point $p \in V$ lies in some level set $F_a$ defined by \eqref{level-sets} and we have $\partial= (1+a)S$ in restriction to $F_a$. Therefore, $\mathbb{T}_p\subset L$ is a $n-$dimensional quasi-periodic torus with frequencies $(1+a)\omega_1,\ldots, (1+a)\omega_n$. 
\end{proof}

\section*{Acknowledgement}
The second author is partially supported by the ANR Project NonSper ANR-23-CE40-0028 Une licence Creative Commons Attribution (CC-BY).  The first author has also been partially supported by the mobility program ``Ayudas para estancias breves en el desarrollo de tesis doctorales. Convocatoria 2025" of the University of Valladolid and by the project PID2022-139631NB-I00 funded by the Agencia Estatal de Investigación - Ministerio de Ciencia e Innovación during the realization of this work. For the purpose of Open Access, the authors have applied a CC-BY public copyright licence to any Author Accepted Manuscript (AAM) version arising from this submission.

\end{document}